\documentclass[12pt]{amsart}

\DeclareMathOperator{\cof}{\cof}

\usepackage[english]{babel}
\usepackage[utf8]{inputenc}
\usepackage{amsmath}
\usepackage{amssymb}
\usepackage{amsfonts}
\usepackage{amsthm}
\usepackage{mathrsfs}
\usepackage[all]{xy}
\usepackage[pdftex]{graphicx}
\usepackage{color}
\usepackage[pdftex,dvipsnames]{xcolor} 
\usepackage{cite}
\usepackage{url}
\usepackage{indent first}
\usepackage[labelfont=bf,labelsep=period,justification=raggedright]{caption}
\usepackage[english]{babel}
\usepackage[utf8]{inputenc}
\usepackage[colorlinks=true, allcolors=blue]{hyperref}
\usepackage[colorinlistoftodos]{todonotes}
\usepackage{tkz-fct}
\usepackage{mathdots}
\usepackage{comment}
\usepackage{float}

\usepackage{algpseudocode}
\usepackage{mathrsfs}
\usepackage{bbm}
\usepackage{dsfont}
\usepackage{theoremref}
\usepackage[dvipsnames]{xcolor}

\usetikzlibrary{automata,positioning}

\theoremstyle{plain}

\usepackage{amsmath, amsthm, amssymb}
\usepackage{hyperref}
\usepackage{mathrsfs}
\usepackage{asymptote}
\usepackage{graphicx}
\usepackage{theoremref}

\theoremstyle{plain}
\newtheorem{theorem}{Theorem}
\newtheorem{lemma}{Lemma}
\newtheorem{proposition}{Proposition}

\theoremstyle{definition}

\numberwithin{equation}{section}
\numberwithin{lemma}{section}
\numberwithin{proposition}{section}

\theoremstyle{remark}
\newtheorem*{remark}{Remark}

\usepackage[margin=1.0in]{geometry}

\begin{document}

\title[Quadratic Uniformity]{Improved Quadratic Gowers Uniformity for the M\"obius Function}
\author{James Leng}
\address[James Leng]{Department of Mathematics\\
     UCLA \\ Los Angeles, CA 90095, USA.}
\email{jamesleng@math.ucla.edu}
\maketitle

\begin{abstract}
We demonstrate that 
$$\|\mu\|_{U^3([N])} \ll_{A}^{\text{ineff}} \log^{-A}(N)$$
$$\|\Lambda - \Lambda_Q\|_{U^3([N])} \ll_{A}^{\text{ineff}} \log^{-A}(N)$$
for any $A > 0$ where $\Lambda_Q$ is an approximant to the von Mangoldt function and will be defined below, improving upon a bound of Tao-Ter\"av\"ainen (2021). As a consequence, among other things, we have the following:
$$\mathbb{E}_{x, y \in [N], x + 3y \in [N]} \Lambda(x)\Lambda(x + y)\Lambda(x + 2y)\Lambda(x + 3y) = \mathfrak{S} + O_A(\log^{-A}(N))$$
where $\mathfrak{S}$ is the singular series for the configuration $(x, x + y, x + 2y, x + 3y)$. In fact, we show that
$$\|\mu - \mu_{Siegel}\|_{U^3([N])} \ll \exp(-O(\log^{1/C}(N)))$$
$$\|\Lambda - \Lambda_{Siegel}\|_{U^3([N])} \ll \exp(-O(\log^{1/C}(N)))$$
where $\mu_{Siegel}$ and $\Lambda_{Siegel}$ are approximants of $\mu$, and $\Lambda$, respectively, representing the Siegel zero contribution of $\mu$ and are defined in the above article. To do so, we use an improvement of the $U^3$ inverse theorem due to Sanders and we follow the approach of Green and Tao (2007), opting to use the ``old-fashioned" approach to equidistribution on two-step nilmanifolds which was also considered by Green and Tao (2017), and by Gowers and Wolf (2010). To the author's knowledge, this is the first time that quadratic Fourier analysis over $\mathbb{Z}/N\mathbb{Z}$ has achieved quasi-polynomial type bounds in applications. 
\end{abstract}

\section{Introduction}
In 2010, Green and Tao \cite{GT1} (with the contribution of work from Green-Tao \cite{GT2} and Green-Tao-Ziegler \cite{GTZ1}) generalized the argument for Vinogradov's sum of three primes theorem for arbitrary linear systems, thus firmly establishing the field of \textit{higher order Fourier analysis} as a generalization of the Hardy-Littlewood circle method. Prior to the work of Manners \cite{M1} and Tao-Ter\"av\"ainen \cite{TT}, the error terms for the asymptotics obtained in \cite{GT1} were ineffective. Two new ingredients were needed to make the bounds effective:
\begin{itemize}
    \item[1.] The first was a quantitative $U^{s + 1}$ inverse theorem, obtained by Manners \cite{M1} in 2018.
    \item[2.] The second was a way to bypass Siegel's theorem before applying the $U^{s + 1}$ inverse theorem, achieved by Tao-Ter\"av\"ainen in 2021 by finding a suitable approximant to the M\"obius function $\mu_{Siegel}$ and estimating the Gowers norm of $\mu - \mu_{Siegel}$. This is important because the loss of powers of logarithms whilst using Vaughan-type decompositions and the fact that the dimension of the nilmanifold is unbounded meant that in order to obtain quantitative bounds, one needs to prove bounds strictly better than strongly-logarithmic, which Siegel's theorem doesn't allow.
\end{itemize}
Using these new tools, Tao and Ter\"av\"ainen \cite{TT} were able to prove the following:
\begin{theorem}\thlabel{TT1}
If $\mu$ and $\Lambda$ are the M\"obius function and von Mangoldt function, respectively, and $\mu_{Siegel}$ and $\Lambda_{Siegel}$ is as in \cite{TT}, $\Lambda_Q$ is as defined below in Section 2 and $N \ge 10$, we have for each $A > 0$
\begin{itemize}
\item[1.] $\|\mu\|_{U^s([N])}, \|\Lambda - \Lambda_{Q}\|_{U^s([N])} \ll (\log\log(N))^{-c}$
\item[2.] $\|\mu_{Siegel}\|_{U^s([N])}, \|\Lambda_{Siegel} - \Lambda_Q\|_{U^s([N])} \ll q_{Siegel}^{-c} \ll_{A}^{\text{ineff}} \log^{-A}(N)$
\item[3.] $\|\mu - \mu_{Siegel}\|_{U^s([N])}, \|\Lambda - \Lambda_{Siegel}\|_{U^s([N])} \ll (\log\log(N))^{-c}$
\item[4.] $\|\mu - \mu_{Siegel}\|_{U^2([N])}, \|\Lambda - \Lambda_{Siegel}\|_{U^2([N])} \ll \exp(-O(\log^{1/C}(N)))$.
\end{itemize}
Here, $[N] = \{1, \dots, N\}$, $q_{Siegel}$ is the conductor for the character $\chi_{Siegel}$ for the Siegel zero $\beta$, $\ll$ is Vinogradov's notation, $U^s([N])$ is the $U^s$ Gowers norm. (These quantities will be specified in Section 2).
\end{theorem}
\begin{remark}
If there is no Siegel zero, then the $\mu_{Siegel} = 0$, $\Lambda_{Siegel} = \Lambda_Q$, and $q_{Siegel}$ disappears.
\end{remark}
We shall prove the following:
\begin{theorem}\thlabel{maintheorem}
There exists a large constant $C \gg 1$ such that for all sufficiently large $N$,
$$\|\Lambda - \Lambda_{Siegel}\|_{U^3([N])} \ll \exp(-O(\log^{1/C}(N)))$$
$$\|\mu - \mu_{Siegel}\|_{U^3([N])} \ll \exp(-O(\log^{1/C}(N))).$$
\end{theorem}
Tao and Ter\"av\"ainen remark that if $s = 3$, it's possible that one can improve (3) of \thref{TT1} if one can improve the $U^3$ inverse theorem to make its dimension bounds $\log^{O(1)}(1/\delta)$ or overcome the induction of dimensions procedure that occurs with the quantitative equidistribution on two-step nilmanifolds. The purpose of this note is to demonstrate that existing tools of additive combinatorics are able to achieve both. \\\\
To achieve the first point, we use a quasi-polynomial $U^3$ inverse theorem using the work of Sanders \cite{S}. Crucially, we are able to replace the $\delta^{-O(1)}$ dimension bound of the $U^3$ inverse theorem of Green and Tao with $\log^{O(1)}(1/\delta)$. For completeness sake we shall provide a proof of this fact in Appendix A. To achieve the second point, we use the equidistribution theory of locally quadratic forms on Bohr sets. As we will explain in Section 1.2 below, bounds lost here are only single exponential in the dimension of the nilmanifold. \\\\
To apply our results to the von Mangoldt function we must avoid using the transference principle since the $W$-trick simply does not allow bounds better than $\log^{-c}(N)$. Instead, we opt to apply the $U^3$ inverse theorem directly. One can prove a similar result using a ``smooth approximant" to the von Mangoldt function $\Lambda_R$, if one uses a version of the $U^3$ inverse theorem with moment bound hypothesis instead of one with pointwise bound hypothesis, which we shall demonstrate in Appendix $A$. For this paper, though, the author ultimately found it more convenient to follow the argument in \cite{TT} where only a $U^3$ inverse theorem with a pointwise bound hypothesis is necessary.
\subsection{General Strategy and Relation to Previous Work}
Our general strategy is to prove an inverse-type theorem of the following form
\begin{theorem}
For each $N> N_0 > 0$ with $N_0$ effective, the following holds: let $\delta > 0$ and suppose
$$\|\mu - \mu_{Siegel}\|_{U^3([N])} \ge \delta$$
and
$$\|\Lambda - \Lambda_{Siegel}\|_{U^3([N])} \ge \delta.$$
Then either $\delta \le \exp(-O(\log^{1/C}(N)))$ or there exists $\epsilon > (\delta/\log(N))^{O(\log(1/\delta))^{O(1)}}$ such that
$$\|\mu - \mu_{Siegel}\|_{U^2([N])} \ge \epsilon$$
$$\|\Lambda - \Lambda_{Siegel}\|_{U^2([N])} \ge \epsilon.$$
\end{theorem}
From there, since we have quasi-polynomial upper bounds on $\|\mu - \mu_{Siegel}\|_{U^2([N])}$ and $\|\Lambda - \Lambda_{Siegel}\|_{U^2([N])}$, we will be able to deduce that $\delta$ is quasi-polynomial in $N$. Since $\delta$ will be quasi-polynomial in $N$, we will often times omit the $\log(N)$ terms in the bounds since $\delta/\log(N) \gg \delta^{O(1)}$ anyways. \\\\
To prove this inverse theorem, we use a type I/type II decomposition as used in previous results of this form, e.g., \cite{GT2}, \cite{GT4}, \cite{TT}. The use of the circle method along with type I and type II decompositions to prove similar results was first used by Vaughan. See \cite{V} for an introduction to related methods. Of these previous works, this paper bears the most resemblance to \cite{GT4}, and thus many of the details of the proof are similar to those in \cite{GT4}. The approach of \cite{GT4} uses the equidistribution theory of quadratic phases on Bohr sets. In this setting, one seemingly has visibly worse versions of the equidistribution theory over a group and we are often left with an inequality of the following form (of which a similar inequality occurs in the type II case):
\begin{equation}
\mathbb{E}_{\ell \in [L, 2L], m \in [M, 2M]} 1_B(\ell m) e(\phi(\ell m)) \gg \delta^{O(1)} \tag{1}
\end{equation}
where $B = B(S, \rho)$ is a Bohr set and $\phi: B(S, 64\rho) \to \mathbb{R}/\mathbb{Z}$ is a locally quadratic form. In our case, the $U^3$ inverse theorem gives us that $|S| = \log^{O(1)}(1/\delta)$ and $\rho \gg \delta^{O(1)}$. One can in theory use the geometry of numbers of write everything in local coordinates (i.e. passing to a generalized arithmetic progression) as done in \cite{GW}, but it becomes unclear what happens when one multiplies two elements of translated generalized arithmetic progressions. Thus, we use a coordinate-free approach to analyze the exponential sum which is based on the following two crucial observations of \cite{GT4}:
\begin{itemize}
\item[1.] Applying a van der Corput type lemma (or Cauchy-Schwarz) to the sum over $\ell \in [L, 2L]$ in (1) resembling the form
$$\mathbb{E}_{\ell \in [L, 2L], \ell' \in [0, L], m \in [M, 2M]} \nu(\ell')1_B(\ell m) 1_B((\ell + \ell' )m) e(\phi((\ell + \ell')m) - \phi(\ell m)) \ge \delta^{O(1)}$$
where $\nu(\ell') = \max(1 - \ell'/L, 0)$. The first observation of \cite{GT4} is that we may Fourier approximate $1_B$ and $\nu$ which would give us an oscillatory sum, at the cost of replacing $\delta$ with $\delta^{O(|S|)^{O(1)}}$. Strictly speaking, we can't use $1_B$ since the Fourier approximation to $1_B$ has errors that are only small in $L^1([N])$. Thus, we must approximate $1_B$ with a smooth function $\psi$ \emph{before} we pass to a type I or type II sum and then Fourier approximate $\psi$. This approximation has errors small in $L^\infty([N])$ and can be used.
\item[2.] The second observation is that if $a$ and $b$ are chosen so that for all $w \in [UV]$ with $U, V > 0$, $wab$ lies in the Bohr set $B(S, \rho)$ and if $(\ell + ua)(m + vb)$ also lie in $B(S, \rho)$ for enough values of $u$ and $v$, then $\partial_{(u_0, v_0), (u_1, v_1)} \phi((\ell + ua)(m + vb))$ is a degree four multivariate polynomial in $u_0, u_1$ and $v_0, v_1$. This is specified in \thref{quartic}.
\end{itemize}
The second point tells us that we should average over progressions of common difference $a$ and common difference $b$ instead of averaging over $\ell$ and $m$. This can be done via a pigeonhole type argument. It is important to note that since we want to average over $a$ and $b$ in a range large enough to exhibit cancellation of the phase that this only works in the type II case where $\ell$ and $m$ are both sufficiently large (e.g., larger than $\delta^{-O(|S|)^{O(1)}}$). The first point tells us that once we have obtained the our degree four multivariate polynomial, we can then Fourier approximate the the remaining undesired terms to obtain an oscillatory sum over a polynomial phase. This can then be estimated using standard techniques. For this article, by performing a Fourier approximation, we lose bounds of \textit{at most} $\delta^{-O(|S|)^{O(1)}}$. These are okay factors to lose since these factors will still be quasi-polynomial in $N$. Thus, we can turn a less well-behaved equidistribution estimate over locally quadratic forms to a well behaved quadratic equidistribution estimate at an acceptable cost. \\\\
Ultimately, we will be able to show that $\phi$, roughly speaking, exhibits one-step behavior in some set of the form
\begin{equation}
\bigcup_{b \in \mathcal{D}} B_b \tag{2}
\end{equation}
where $B_b = \{x \in B: x \equiv 0 \pmod{b}\}$ and $\mathcal{D}$ is some large enough set. Our aim is to show that this set is large enough so that we can propagate $\phi$ to have one-step behavior in some shrunken version of $B$. Unfortunately, here, the approach of \cite{GT4} breaks down. In \cite[Proposition 11.1]{GT4}, the authors use a second moment method, relying on the $O(|S|)^{\mathrm{th}}$ moments of the divisor function, resulting in losses double exponential in $|S|$, which are unacceptable bounds for us to lose. If one settles for worse bounds with low divisor moment estimates, then the lower bound of (2) is something of the size $\rho^{O(|S|)}|B|$, which is far too small for us to conclude anything about $\phi$. This is where this paper differs from \cite{GT4}. In the type II case, the set $\mathcal{D}$ is an interval, so we can restrict to the set of prime numbers since there are enough prime numbers in $\mathcal{D}$ so that we only lose factors of $\log^{O(1)}(N)$ from passing to the prime numbers. There, instead of using the divisor function, we use the \textit{prime divisor function}, which has very good moment bounds. In the type I case, we don't have a nice description of $\mathcal{D}$. Instead, we restrict the type I/type II decomposition so that the type I sum appears as
$$\sum_{d \le R} a_d 1_{d | n} 1_{[N']}(n)$$
with $R$ is quasi-polynomial in $N$ (i.e. $R$ is extremely small). Here, $a_d$ are divisor function bounded coefficients which can be controlled well. Another difference from \cite{GT4} is the use of $\mu_{Siegel}$ and $\Lambda_{Siegel}$ as in \cite{TT}. It is shown in \cite{TT} that $\mu_{Siegel}$ and $\Lambda_{Siegel}$ can be decomposed into type I and twisted type I sums of the form
$$\sum_{d \le \exp(O(\log^{1/2}(N)))} a_d 1_{d | n} 1_{[N']}(n)$$
$$\sum_{d \le \exp(O(\log^{1/2}(N)))} a_d 1_{d | n} 1_{[N']}(n)\chi_{Siegel}(n/d)$$
respectively (with $N' \le N$). In this case, $d$ can be too large for the type I case to handle, so we must handle cases of when $d \le R$ and $d > R$ separately. In the former case, we handle the type I sum the same way we would handle the type I sum for $\mu$. In the latter case, we handle the type I case like we would handle the type II case for $\mu$ since we are in the range of where $\ell$ and $m$ are sufficiently large. \\\\
There is work done by Le and Bienvenu \cite{BL} proving an analogous result in the case $\mathbb{F}_p[x]$. There, an analogous result to \thref{quartic} no longer appears to be true for $[U]$ and $[V]$ long intervals (or rather, when $[U]$ and $[V]$ are subspaces $\{f \in \mathbb{F}_p[x]: \deg(f) \le U\}$ and $\{f \in \mathbb{F}_p[x]: \deg(f) \le V\}$, respectively). Thus, in the type II case, unlike us, \cite{BL} is no longer able to exploit cancellation in \emph{both} the $[L, 2L]$ and $[M, 2M]$ sums in a sum similar to (1) and a set which plays a similar role as $\mathcal{D}$ one has to deal with is considerably more complex. It is shown after taking vertical and horizontal differences that using the bilinear Bogolyubov theorem, one may pass to a bilinear Bohr set, which was well-behaved enough for \cite{BL} to obtain their result. Although such a tool exists over general Abelian groups as shown by Mili\'cevi\'c in \cite{Mi}, no such tool is necessary since $\mathcal{D}$ in the type II case is an interval. For another instance where it seems necessary to exploit cancellation in both the $[L, 2L]$ and $[M, 2M]$ sums in a type II sum, see \cite{MSTT}.
\subsection{Remark on the Quantitative Equidistribution Theory of Two-Step Nilsequences}
Here, we will briefly explain why the equidistribution estimates considered here are single exponential in dimension and thus quantitatively speaking significantly better than \cite{GT3}. The reason is that we invoke the geometry of numbers which allows us to eliminate many rational relations at once. In \cite{JT}, the authors construct, given a locally quadratic form on a Bohr set, a nilsequence on (essentially\footnote{The explicit ambient nilmanifold that the construction of \cite{JT} lies on is a three-step nilmanifold, but the nilsequence that \cite{JT} constructs lies on a coset of a two-step nilmanifold.}) a two-step nilmanifold that corresponds to it. Viewing the equidistribution theory of locally quadratic forms in this form, we see that the conclusion of \cite[Proposition 4.11]{GT5} or \cite[Corollary 6.2]{GW} implies that the locally quadratic form has roughly rational coefficients when written in local coordinates. This can be combined with Vinogradov's Lemma (e.g. \thref{Vinogradov}) to show that the degree one and degree two components of $\text{Poly}_{\le 2}(\mathbb{R}^{d})$ in the construction of the nilsequence (denoted $\Phi(h, \theta)(n)$ in \cite{JT}) are roughly rational, which implies that, modulo some rational relation, the nilsequence lies on (a coset of) an abelian nilmanifold. \\\\
This process is more efficient than Green-Tao's original approach of locating a single horizontal character to quotient the nilpotent Lie group by to lower the dimension, since the horizontal character depends on the complexity of the nilmanifold, which can increase when we take a quotient by the horizontal character. This leads to losses double exponential in dimension.

\subsection{Organization of the Paper}
In Section 2, we define the notation used in the paper. In Section 3, we prove Fourier approximation results in order to ``automate" the Fourier approximation heuristic in the above discussion that will appear later in the argument. In Section 4, we apply the $U^3$ inverse theorem, take a type I/II decomposition, and reduce the theorem to an equidistribution estimate, ultimately reducing the problem to \thref{type1type2}. The proof of \thref{type1type2} will occupy Sections 5, 6, and 7. In Section 5, we analyze the type I and twisted type I sum. In Section 6, we analyze the type II sum. In Section 7, we use the conclusions of Sections 5 and 6 to finish \thref{type1type2}. In Section 8, we apply the main result to obtain asymptotic estimates in combinatorial number theory. \\\\
In Appendix A, we provide a proof of the $U^3$ inverse theorem over $\mathbb{Z}/N\mathbb{Z}$ with polylogarithmic dimension bounds, using the work of Sanders \cite{S} and Green-Tao \cite{GT7}. In Appendix B, we record a few classical results in equidistribution theory and diophantine approximation. Finally, in Appendix C, we recall some basic properties of the Gowers uniformity norms.
\subsection{Acknowledgements}
We would like to thank Terence Tao for valuable guidance, helpful discussions, and for numerous comments on previous drafts. We would also like to thank Joni Ter\"av\"ainen and Thomas Bloom for helpful and encouraging comments comments and corrections, and Borys Kuca for ecouragement. Some of these discussions happened while the author was visiting the IAS during the conferences \emph{Workshop on Additive Combinatorics and Algebraic Connections} and \emph{Workshop on Dynamics, Discrete Analysis and Multiplicative Number Theory} and we would like to thank the organizers of these conferences. This research is supported by the NSF Graduate Research Fellowship Grant No. DGE-2034835.

\section{Notation}
We shall use standard Vinogradov notation of $A \ll B$ or $A = O(B)$ if $A \le CB$ for some constant $C$. We denote $e(x) := e^{2\pi i x}$. Let $[N] = \{1, \dots, N\}$, $[M, N] = \{M, M + 1, M + 2, \dots, N\}$ for $M \le N$ integers, and given a set $A$, $\mathbb{E}_{n \in A} = \frac{1}{|A|} \sum_{n \in A}$, $\|x\|_{\mathbb{R}/\mathbb{Z}}$ denotes the distance from $x$ to its nearest integer. We also denote for $Q \in \mathbb{R}$ positive,
$$\|x\|_{Q, \mathbb{R}/\mathbb{Z}} := \inf_{|q| \le Q, q \in \mathbb{Z}} \|qx\|_{\mathbb{R}/\mathbb{Z}}.$$
Given a finite abelian group $G$ and a function $f\colon G \to \mathbb{C}$, let 
$$\|f\|_{U^s(G)}^{2^s} := \mathbb{E}_{n, h_1, h_2, \dots, h_s \in G} \prod_{\omega \in \{0, 1\}^s} C^{|\omega|}f(x + \omega \cdot h)$$
be the $U^s$ Gowers uniformity norm. We shall denote
$$\|f\|_{U^s([N])} := \frac{\|f1_{[N]}\|_{U^s(\mathbb{Z})}}{\|1_{[N]}\|_{U^s(\mathbb{Z})}}$$
where for a finitely supported function $h: \mathbb{Z} \to \mathbb{C}$,
$$\|h\|_{U^s(\mathbb{Z})} := \sum_{n, h_1, h_2, \dots, h_s \in \mathbb{Z}} \prod_{\omega \in \{0, 1\}^s} C^{|\omega|}f(x + \omega \cdot h).$$
We refer the reader to Appendix C, \cite{T2}, and \cite{TV} for more properties of Gowers uniformity norms. For $G$ locally compact and abelian and $f\colon G \to \mathbb{C}$, its Fourier transform $\hat{f}\colon \widehat{G} \to \mathbb{C}$, where $\widehat{G}$ its the pontryagin dual of $G$, will be defined as
$$\hat{f}(\xi) := \int_{G} e(\xi \cdot x) f(x) d\mu(x)$$
for $\xi \in \widehat{G}$ where $\mu$ is the Haar measure of $G$. If $G$ is finite, then
$$\hat{f}(\xi) = \mathbb{E}_{x \in G} e(\xi \cdot x) f(x).$$
Let $\Lambda$ be the von Mangoldt function, $\mu$ the M\"obius function, and $N > 0$ a large integer. Let $Q = \exp(\log^{1/10}(N))$. Let $\beta$ be a possible Siegel zero of level $Q$ with conductor $q_{Siegel} \le Q$. Following \cite{TT}, we define
$$\mu_{Siegel} := (1_{n | P(Q)} \mu(n)) * (\alpha n^{\beta - 1} \chi_{Siegel}(n)1_{(n, P(Q)) = 1})$$
$$\Lambda_Q := \frac{P(Q)}{\phi(P(Q))}1_{(n, P(Q)) = 1}$$
$$\Lambda_{Siegel} := \Lambda_Q(1 - n^{\beta - 1} \chi_{Siegel}(n))$$
where $P(Q) := \prod_{p < Q, p \text{ prime}} p$, $\chi_{Siegel}$ is the character for the Siegel zero $\beta$, and
$$\alpha := \frac{1}{L'(\beta, \chi_{Siegel})}\prod_{p < Q} \left(1 - \frac{1}{p}\right)^{-1}\left(1 - \frac{\chi_{Siegel}(n)}{p^\beta}\right)^{-1}.$$
For a locally compact abelian group $G$, and a subset $S \subseteq \hat{G}$, the Bohr set of $S$ with radius $\rho$, denoted $B(S, \rho)$, consists of $\{x \in G: \|\alpha \cdot x\| < \rho \text{ } \forall \alpha \in S\}$. If $G$ is finite, Bohr sets satisfy the following properties:
\begin{itemize}
    \item $\frac{4}{|S|} \ge |B(S, \rho)| \ge \rho^{-|S|}N$
    \item $|B(S, 2\rho)| \le 4^{|S|}|B(S, \rho)|$
\end{itemize}
A Bohr set is \textit{regular} if whenever $\epsilon \le \frac{1}{100|S|}$, 
$$|B(S, \rho)| (1 - 100|S|\epsilon) \le |B(S, \rho (1 + \epsilon))| \le |B(S, \rho)|(1 + 100|S|\epsilon).$$
It was shown by Bourgain in \cite{B} that regular Bohr sets are ubiquitous in the sense that there exists $\rho' \in [\rho/2, \rho]$ such that $B(S, \rho')$ is regular. Unless specified otherwise, all Bohr sets will be assumed to be regular. A useful property of regular Bohr sets is that it allows us to localize to a smaller Bohr set with small error:
\begin{lemma}[Localization lemma]\thlabel{localization}
Let $B(S, \rho)$ be a regular Bohr set, $f\colon G \to \mathbb{C}$ a function, and $A \subseteq B(S, \epsilon \rho)$ a set. Then
$$|\mathbb{E}_{x \in B(S, \rho)} f(x) - \mathbb{E}_{x \in B(S, \rho)}\mathbb{E}_{y \in x + A} f(y)| \le 200|S| \epsilon.$$
\end{lemma}
\begin{proof}
This follows from the fact that $B(S, \rho)$ and $y + B(S, \rho)$ differs by at most $200 |S| \epsilon$ elements whenever $y \in B(S, \epsilon \rho)$.
\end{proof}
We will need to work with smooth approximants to Bohr sets $1_B$ supported in $B$, which will denoted $\psi$ throughout the article. We will define the ``norms"
$$\|n\|_S = \sup_{\alpha \in S} \|\alpha n \|_{\mathbb{R}/\mathbb{Z}}.$$
In addition, we will assume that $1/N \in S$ so we have the following inequality:
$$|n| \le N\|n\|_S$$
whenever $n \in [-N/2, N/2]$. This assumption will turn out to be harmless because one can always localize to a small Bohr set by using the pigeonhole principle and regularity properties of the Bohr set. \\\\
Given a Bohr set $B(S, 16\rho)$, a \emph{locally quadratic form} $\phi\colon B(S, 16\rho) \to \mathbb{R}/\mathbb{Z}$ is a function that satisfies
$$\partial_{h_1, h_2, h_3} \phi(x) = 0$$
whenever $x + \omega \cdot h \in B(S, 16\rho)$ for all $\omega \in \{0, 1\}^3$ where $\partial_h f(x) = f(x + h) - f(x)$ is the discrete derivative and $\partial_{h_1, h_2, h_3}f(x) = \partial_{h_1}(\partial_{h_2}(\partial_{h_3} f(x)))$. We may define a locally bilinear form associated to it $\phi'': B(S, \rho) \times B(S, \rho) \to \mathbb{R}/\mathbb{Z}$ associated to $\phi$ by defining 
$$\phi''(a, b) := \phi(a + b) - \phi(a) - \phi(b) + \phi(0).$$ 
In fact, one can show that $\phi''(a, b) = \phi(n + a + b) - \phi(n + a) - \phi(n + b) + \phi(n)$ whenever $n \in B(S, \rho)$. To check that $\phi''$ is indeed bilinear, we see that 
\begin{align*}
\phi''(a + c, b) &= \phi(a + b + c) - \phi(a + c) - \phi(b) + \phi(0) \\
&= \phi''(b, c) + \phi(a + b) + \phi(a + c) - \phi(a) - \phi(a + c) - \phi(b) + \phi(0)\\
&= \phi''(a, b) + \phi''(c, b)
\end{align*}
from the identity
$$\phi(a + b + c) - \phi(a + b) - \phi(a + c) + \phi(a) = \phi''(b, c).$$
We would like to remark that although the statement of the local $U^3$ inverse theorem only gives correlation on $B(S, \rho)$ of the function with a quadratic phase defined translation of $B(S, \rho)$, we would like for technical reasons for $\phi$ to be defined on a slightly larger Bohr set $B(S, 64\rho)$. This is possible via a small modification of the argument of the inverse theorem which we shall state in \thref{sandersinversetheorem}.
\section{Fourier Complexity}
The purpose of this section is to make rigorous and automate the ``Fourier approximation" heuristic argument as mentioned in the introduction. We recall from discussion of the introduction that while analyzing oscillatory sums over objects that aren't quite abelian groups or arithmetic progressions, we may sometimes end up with cutoff terms while invoking Cauchy-Schwarz. The example mentioned in the introduction was
$$\mathbb{E}_{\ell \in [L, 2L], m \in [M, 2M]} \psi(\ell m) e(\phi(\ell m)) \gg \delta^{O(1)}.$$
In a similar looking sum appearing later in our argument, it becomes desirable to eliminate the term $\psi$. The method used here to eliminate such terms is to Fourier approximate such a $\psi$. \\\\
A useful notion in additive combinatorics is \textit{Fourier Complexity}, which measures how many Fourier phases $n \mapsto e(\alpha n)$ captures the behavior of $n \mapsto f(n)$. Specifically, we say that a function $f\colon [N] \to \mathbb{C}$ has \emph{$L^p$ $\delta$-Fourier complexity} at most $M$ if we can write
$$f = \sum_{i = 1}^k a_i e(\alpha_i n) + g$$
where $\|g\|_{L^p} \le \delta$ and $\sum_{i = 1}^k |a_i| \le M$. If $p = \infty$, we shall refer to $L^p$ $\delta$-Fourier complexity as just $\delta$-Fourier complexity and if $\delta = 0$, we will refer $\delta$-Fourier complexity as \emph{Fourier complexity}. 
\begin{lemma}[Fourier/Fejer Expansion lemma]\thlabel{FourierExpansion}
Let $f\colon \mathbb{T}^d \to \mathbb{C}$ be a continuous function with Lipschitz norm at most $L$, meaning that $\|f\|_{L^\infty(\mathbb{T}^d)} + \|f\|_{Lip(\mathbb{T}^d)} \le L$. Then we may write
$$f = \sum_{i = 1}^k a_i e(n_i x) + g$$
where $\sum_i^k |a_i| \le C^{d^2}L\delta^{-2d^{2} - d}$ and $\|g\|_\infty \le 3\delta$.
\end{lemma}
\begin{proof}
Let $\phi\colon \mathbb{R} \to \mathbb{R}_{\ge 0}$ be a smooth compactly supported function supported in $[-1, 1]$ and with integral one. Let $Q_\delta(x) = \prod_{i = 1}^d \delta^{-1}\phi(x_i/\delta)$ and let $K = Q_\delta * Q_\delta$ be a Fejer-type kernel. Since $|\hat{\phi}(\xi)| \ll_k |\xi|^{-k}$,
$$|\hat{K}(\xi)| \le C^d \delta^{2d} \xi^{-2}$$
for some constant $C$ it follows that the Fourier coefficients of $K$ larger than $M$ contributes at most $C^d M^{-1}\delta^{2d}$. In addition, for $M = C^{-d}\delta^{-2d - 1} \|f\|_{Lip}^{-1}$
$$\|f - f * K\|_\infty \le \int |f(x) - f(y)|K(x - y) dy = \int z K(z) dz \le 2\delta$$
since $K$ has integral one and is supported on $|x| \le 2\delta$. Set
$$h(x) = \sum_{k \in \mathbb{Z}^d: |k| \le M} \hat{f}(k) \hat{K}(k) e(kx).$$
Then by Fourier inversion formula, it follows that
$$\|h - f * K\|_\infty \le \delta.$$
Thus, $\|h - f\|_\infty \le 3\delta$. The sum of the Fourier coefficients of $H$ is at most $L C^{-d^2}\delta^{-2d^2 - d}$.
\end{proof}
Thus, if $f(n) = F(n\alpha)$ where $f\colon \mathbb{T}^d \to \mathbb{C}$ is Lipschitz continuous and has Lipschitz norm at most $L$, then $f$ has $\delta$-Fourier complexity at most $O(L/\delta)^{O(d^2)}$. Such a sequence $(F(n\alpha))_{n \in [N]}$ is known as a \textit{degree one nilsequence}. We say that a sequence $g \colon [N] \to \mathbb{C}$ is an $\epsilon$-\textit{approximate degree one nilsequence} if there exists a degree one nilsequence $F(\alpha n)$ such that 
$$\mathbb{E}_{n \in [N]} |F(\alpha n) - g(n)| \le \epsilon.$$
A frequent example of an approximate degree one nilsequence in this article is an indicator function of a regular Bohr set. Regularity is not strictly necessary but not assuming regularity introduces some unnecessary complications (that can be dealt with using similar arguments to those in \cite[Appendix E]{GTZ2}). Let $B(S, \rho) = \{x \in [N]: \|\alpha x\| \le \rho \text{ }\forall \alpha \in S\}$.
\begin{lemma}\thlabel{BohrFourier}
Suppose $B(S, \rho)$ is a regular Bohr set. Let $\delta < \rho/1000$. Then $1_{B(S, \rho)}$ is an $\delta$-approximate degree one nilsequence with Lipschitz norm at most $1000|S|/\delta$, and hence has $\epsilon$-Fourier complexity at most $\left(\frac{1000|S|}{\delta \epsilon}\right)^{O(|S|^2)}$.
\end{lemma}
\begin{proof}
Let $\xi\colon [N] \to (\mathbb{R}/\mathbb{Z})^{|S|}$ via $n \mapsto (\alpha \cdot n)_{\alpha \in S}$. Thus, a candidate Lipschitz function is some function $\varphi$ that is $\delta/100|S|$ close $1_{\|x\|_{\mathbb{R}^{|S|}/\mathbb{Z}^{|S|}} < \rho}$ in $L^1$ norm which has support in $\|x\|_{\mathbb{R}^{|S|} /\mathbb{Z}^{|S|}} \le \rho + \delta/|S|$. We must show that this implies that $\varphi(\xi) - 1_{B(S, \rho)}$ has small $L^1$ norm. Since $B(S, \rho)$ is regular, it follows that
$$\|\varphi(\xi(\cdot)) - 1_{B(S, \rho)}(\cdot)\|_{L^1[N]} \le \delta.$$
Such a Lipschitz function $\varphi$ can be chosen to have Lipschitz norm $1000|S|/\delta$. The rest of the lemma thus follows from \thref{FourierExpansion}.
\end{proof}
We shall also prove that given $\phi$ a locally quadratic form, and $\phi''$ its associated bilinear form with $\|\phi''\|_{\mathbb{R}/\mathbb{Z}} \le \frac{1}{10}$ that $1_{B(S, \rho/16)}(n)e(\phi(n) - \frac{1}{2}\phi''(n, n))$ has bounded Fourier complexity. First, we show that $\phi(n) - \frac{1}{2}\phi''(n, n)$ is locally affine linear. To see this, observe that $\phi(a + b) - \phi(a) - \phi(b) = \phi''(a, b)$, so $(\phi - \frac{1}{2}\phi'')(a + b) - (\phi - \frac{1}{2}\phi'')(a) - (\phi - \frac{1}{2}\phi'')(b)$ is constant where $(\phi - \frac{1}{2}\phi'')(n)$ denotes $\phi(n) - \frac{1}{2}\phi''(n, n)$. Here, we define $\frac{1}{2}\phi''(n, n)$ so that it takes values in $[-1/20, 1/20]$.
\begin{lemma}[Fourier Complexity of $U^2$-Dual Functions]\thlabel{fourieru2dual}
Let $f(x) = \sum_{y} g(x + y)h(y)$ for $g, h$ functions on $\mathbb{Z}/N\mathbb{Z}$. Then $f$ has Fourier complexity at most $\|g\|_{L^2}\|h\|_{L^2}$.
\end{lemma}
\begin{proof}
Applying Cauchy-Schwarz and Plancherel's Theorem, we obtain
$$\sum_{\xi \in \widehat{\mathbb{Z}/N\mathbb{Z}}} |f(\xi)| \le \sum_{\xi} |\hat{g}(\xi) \overline{\hat{h}}(\xi)| \le \|g\|_{L^2}\|h\|_{L^2}.$$
\end{proof}
\begin{lemma}\thlabel{locallylinear}
Let $\ell\colon B(S, \rho) \to \mathbb{R}/\mathbb{Z}$ be a locally linear function and let $\epsilon < 1000 |S|$. Then $1_{B(S, \rho)} e(\ell)$ has $L^1$ $\epsilon$-Fourier complexity at most $(\epsilon \rho)^{-O(|S|)}$.
\end{lemma}
\begin{proof}
Notice that $\ell(x + y) - \ell(y) = \ell(x)$ whenever $x \in B(S, (1 - \epsilon/|S|) \rho)$ and $y \in B(S, \epsilon \rho/|S|)$. In addition, we may write
$$1_{B(S, \rho)}(x) = \mathbb{E}_{y \in B(S, \epsilon \rho/|S|)} 1_{B(S, \rho (1 - \epsilon/|S|))}(x + y)1_{B(S, \epsilon \rho/|S|)}(y) + O_{L^1}(\epsilon).$$
The lemma then follows from \thref{fourieru2dual}.
\end{proof}
Finally, we will state the below lemma, which will be relevant for Section 8:
\begin{lemma}[Fourier Complexity of Convex Sets]\thlabel{fourierconvex}
Let $K \subseteq [-N, N]^d$ be a convex set, which is the restriction of a convex set in $\mathbb{R}^d$ supported in the continuous region $[-N, N]^d$ (note that we are breaking our convention in Section 2 that $[-N, N]$ is discrete). Then we may write $1_K$ is a $L^1$ $\epsilon$-approximate nilsequence in $(\mathbb{Z}/4N\mathbb{Z})^d$ of dimension $d$ and Lipschitz norm $O_d(1/\epsilon)$.
\end{lemma}
\begin{proof}
See \cite[Corollary A.3]{GT1}.
\end{proof}
\section{Reducing to Equidistribution Estimates}
In this section, we shall turn the condition 
$$\|\mu - \mu_{Siegel}\|_{U^3([N])} \ge \delta$$
and 
$$\|\Lambda - \Lambda_{Siegel}\|_{U^3([N])} \ge \delta$$
into a condition that an oscillatory sum is large, ultimately arriving at the hypotheses of \thref{type1type2}. In order to do so, we must first apply the $U^3$ inverse theorem, a task that is done in Section 4.1, and to decompose $\mu, \mu_{Siegel}, \Lambda, \Lambda_{Siegel}$ into type I sums and type II sums, a task that is done in Section 4.2. Informally, the purpose of doing so is to turn a single sum into a double sum or a triple sum. This gives us freedom to swap orders of summation and to Cauchy-Schwarz away undesirable terms while still being able to exploit cancellation in phase. In contrast, with a single sum, ``Cauchy-Schwarzing" away undesirable terms is still possible, but one cannot hope to exploit any cancellation in phase once we do so. 
\subsection{Applying the $U^3$ inverse theorem}
Let $\delta > 0$ with $\delta \ll_A \log^{-A}(N)$ for some large power $A$. Suppose
$$\|\mu - \mu_{Siegel}\|_{U^3([N])} \ge \delta$$
or
$$\|\Lambda - \Lambda_{Siegel}\|_{U^3([N])} \ge \delta.$$
First, we find some prime $P \in [100N, 200N]$, such that by extending $\mu$ and $\mu_{Siegel}$ and $\Lambda$ and $\Lambda_{Siegel}$ to be zero outside $[N]$, we have
$$\|\mu - \mu_{Siegel}\|_{U^3(\mathbb{Z}/P\mathbb{Z})} \gg \delta$$
$$\|\Lambda - \Lambda_{Siegel}\|_{U^3(\mathbb{Z}/P\mathbb{Z})} \gg \delta^{O(1)}.$$
Then by \thref{sandersinversetheorem}, there exists a Bohr set $B(S, \rho)$ and a locally quadratic form $\phi\colon B(S, 64 \rho) \to \mathbb{R}/\mathbb{Z}$ and a translation $h \in [P]$ such that
$$|\mathbb{E}_{n \in B(S, \rho)} (\mu - \mu_{Siegel})(n + h) e(\phi(n))| \gg \delta^{O(1)}$$
where $|S| \le O(\log(1/\delta))^{O(1)}$ and $\rho \ge \exp(-O(\log^{O(1)}(1/\delta)))$ or that
$$|\mathbb{E}_{n \in B(S, \rho)} (\Lambda - \Lambda_{Siegel})(n + h) e(\phi(n))| \gg (\delta/\log(N))^{O(1)}$$
with $|S| \le O(\log(\log(N)/\delta))^{O(1)}$ and $\rho \ge \exp(-O(\log(\log(N)/\delta)^{O(1)}))$ (see the discussion near the end of Section 2 for how we can reduce to $\phi$ being defined on $B(S, 64\rho)$ instead of just $B(S, \rho)$).
For simplicity, we shall often assume below that $h = 0$, for all of our estimates work for any translation of a Bohr set. We do note that the bilinear form $$\phi''(a, b):= \phi(h + a + b) - \phi(h + a) - \phi(h + b) + \phi(h)$$
associated to $\phi$ is defined on $B(S, \rho) \times B(S, \rho)$ even if $\phi$ were defined on a translation of a Bohr set $h + B(S, 64\rho)$ and satisfies
$$\phi''(a, b) = \phi(n + a + b) - \phi(n + a) - \phi(n + b) + \phi(n)$$
for $n \in h + B(S, \rho)$. To prove \thref{maintheorem}, it suffices to show that $\delta$ cannot be larger than $\exp(-O(\log^{1/C}(N)))$ for some $C > 0$. For technical reasons, we shall work with a smooth cutoff $\psi$ instead of $1_B$. Thus, we will choose a smooth cutoff $\psi$ that realizes $1_B$ as a $\delta^{O(|S|)^{O(1)}}$-approximate nilsequence. We shall also work with the variable ``$N$" instead of the variable ``$P$."
\subsection{Type I, Twisted Type I, and Type II Sums}
A \textit{type I sum} is a sum of the form
$$n \mapsto \sum_{d \le R} a_d 1_{d|n}  1_{[N']}(n)$$
where $a_n \ll \tau(n)^k$ for some absolute constant $k$, $R \ge \exp(O(\log^{1/C}(N)))$ is a parameter that is at least quasi-polynomial in $N$, and $\tau(n)$ is the number of divisors of $n$. A \textit{twisted type I sum} is a sum of the form
$$n \mapsto \sum_{d \le R} a_d 1_{d|n} \chi_{Siegel}(n/d) 1_{[N']}(n).$$
Note that if $c$ is a type I sum, and if
$$|\mathbb{E}_{n \in [N]} f(n) c(n)| \ge \delta$$
then
$$\left|\mathbb{E}_{n} \sum_{d \le R} f(n)a_d1_{d | n} 1_{[N']}(n)\right| \ge \delta.$$
The point of these type I sums is that we may interchange sums and use the triangle inequality:
$$\sum_{d \le R} \tau(d)^k \left|\sum_{n \in [N'], d | n} f(n) \right| \gg \delta N.$$
Since we have divisor moment estimates
$$\sum_{d \le R} \frac{\tau(d)^k}{d} \ll \log^{O(1)}(N)$$
it follows from Cauchy-Schwarz that
$$\sum_{d \le R} d \left|\sum_{n \in [N'], d | n} f(n) \right|^2 \gg \frac{\delta^{2} N^2}{\log^{O(1)}(N)}.$$
By dyadic pigeonholing, it follows that there exists $D \le R/2$ such that for $\frac{\delta D}{\log^{O(1)}(N)}$ many $d \in [D, 2D]$, we have
$$\left|\sum_{n \in [N'], d | n} f(n)\right| \gg \frac{\delta N}{D \log^{O(1)}(N)}.$$
Thus, if $f$ were an oscillatory, one turns a correlation with $a$ into an oscillatory sum for which equidistribution theories can apply. A similar calculation occurs for twisted type I sums, obtaining
$$\left|\sum_{n \in [N'], d|n} f(n) \chi(n/d)\right| \gg \frac{\delta N}{D\log^{O(1)}(N)}$$
with $d \in [D, 2D] \le R/2$. \\\\
A \textit{type II sum} is a sum of the form
$$n \mapsto \sum_{d, w > \sqrt{R}, dw = n} a_d b_w.$$
where $a_d, b_w \ll \tau(n)^k$ for some absolute $k$ and as above, $R$ is a parameter that is at least quasi-polynomial in $N$. Now suppose $c$ is a type II sum and that
$$\left|\sum_{n \in [\delta N, N]} f(n)c(n)\right| \ge \delta N.$$
The reason we opt to work with $[\delta N, N]$ will be clearer later in the argument. In hypotheses such as
$$\left|\sum_{n \in [N]} f(n) c(n)\right| \ge \delta N$$
with $c$ either $\mu, \mu_{Siegel}, \Lambda, \Lambda_{Siegel}$, we may, at the cost of shrinking $\delta$ to $\delta^2$, replace $[N]$ with $[\delta N, N]$. Since $\delta \ll_A \log^{-A}(N)$, all of the contribution of the sum for $n \le \delta^2N$ will be smaller than $\delta N$. Substituting the type II sum, we obtain
$$\left|\sum_{n \in [\delta N, N]} \sum_{d, w > \sqrt{R}, dw = n} a_d b_w f(n)\right| \ge \delta N.$$
The point is that we may use Cauchy-Schwarz to turn this type II sum into a type I sum which can be handled similarly as above. Using Cauchy-Schwarz and the divisor moment bound
$$\sum_{d \in [\sqrt{R}, N/\sqrt{R}]} \tau(d)^k \left|\sum_{w \in [\delta N/d, N/d]} f(n)\right| \gg \delta N,$$
we obtain
$$\sum_{\sqrt{R} < d < N/\sqrt{R}}  d \left|\sum_{\delta N/d \le w \le N/d} b_w f(dw)\right|^2 \gg \frac{\delta^2 N^2}{\log^{O(1)}(N)}.$$
Using dyadic pigeonholing, and divisor bounds once more, we find that there exists $D \in [\sqrt{R}, N/\sqrt{R}]$ and $W$ such that
$$\sum_{d \in [D/2, D]} \left|\sum_{w \in [W/2, W]} 1_{[\delta N/d, N/d]}(w)b_w f(dw)\right|^2 \gg \frac{\delta^2 N^2}{D\log^{O(1)}(N)}$$
and $\delta/4 N \le DW \le N$. Fourier approximating $1_{[\delta N/d, N/d]}$ via \thref{FourierExpansion}, and expanding, we find that
$$\sum_{d \in [D/2, D]} \sum_{w, w' \in [W/2, W]} e(\alpha (w - w')) b_w\overline{b_{w'}} f(dw) \overline{f(dw')} \gg \frac{\delta^{O(1)} N^2}{D\log^{O(1)}(N)}.$$
Thus
$$\mathbb{E}_{d \in [D/2, D]} \mathbb{E}_{w, w' \in [W/2, W]} b(w, w') f(dw) \overline{f(dw')} \gg \frac{\delta^{O(1)}}{\log^{O(1)}(N)}$$
for some divisor power bounded function $b$. Thus, by using Cauchy-Schwarz in $d$, we obtain
$$\sum_{d, d' \in [D/2, D]} \sum_{w, w' \in [W/2, W]} f(dw)\overline{f(d'w)f(dw')}f(d'w') \gg \frac{\delta^{O(1)}N^2}{\log^{O(1)}(N)}.$$
If $f$ were an oscillatory phase, this procedure turns a correlation of a type II sum and an oscillatory phase into an oscillatory sum. We have thus proved the following:
\begin{proposition}\thlabel{decomposition}
Let $c(n)$ be a type I sum, $c_1(n)$ be a twisted type I sum and $c_2(n)$ be a type II sum. Suppose $f\colon [N] \to \mathbb{C}$ is a one-bounded function.
\begin{itemize}
    \item Suppose $\left|\sum_{n \in [N]} f(n)c(n)\right| \ge \delta N$. Then there exists $N' \le N$, $D \le R$ such that for $\delta \log^{-O(1)}(N)D$ many elements in $[D/2, D]$ such that
    $$\left|\sum_{n \in [N'/d]} f(dn)\right| \gg \frac{\delta N}{\log^{O(1)}(N)}.$$
    \item Suppose $\left|\sum_{n \in [N]} f(n)c_1(n)\right| \ge \delta N $. Then there exists $N' \le N$, $D \le R$ such that for $\delta \log^{-O(1)}(N)D$ many elements in $[D/2, D]$ such that
    $$\left|\sum_{n \in [N'/d]} f(dn)\chi(n)\right| \gg \frac{\delta N}{\log^{O(1)}(N)}.$$
    \item Suppose $\left|\sum_{n \in [\delta N, N]} f(n)c_2(n)\right| \ge \delta N$. Then there exists $\sqrt{R} \le D \le N/\sqrt{R}$ and $W \in [\delta N/D, N/D]$ such that
    $$\sum_{d, d' \in [D/2, D]} \sum_{w, w' \in [W/2, W]} f(dw)\overline{f(d'w)f(dw')}f(d'w') \gg \frac{\delta^{O(1)}N^2}{\log^{O(1)}(N)}.$$
\end{itemize}
\end{proposition}
In practice, we shall assume that $\delta$ is quasi-polynomial in $N$, so it is much smaller than $\log^{O(1)}(N)$. As a consequence, we may assume that $N' \ge \delta^{O(1)}N$. It is a classical result in analytic number theory that we may decompose the M\"obius function $\mu$ and the von Mangoldt function $\Lambda$ as a combination of type I and type II sums. Tao and Ter\"av\"ainen \cite[Proposition 7.2]{TT} observe that we may decompose $\mu_{Siegel}$ and $\Lambda_{Siegel}$ into a convex combination of type I sums and twisted type I sums plus a negligible error which is quasi-polynomial in size. We record their result below (we note that analogous decompositions for $\mu$ and $\Lambda$ are due to Vaughan, a proof of which can be found in \cite[Chapter 3]{V}):
\begin{lemma}
The functions $\mu, \Lambda$, and $\Lambda_Q$ can be decomposed into a convex combination of type I sums with $R = \exp(O(\log^{1/2}(N)))$, and type II sums up to an $L^1([N])$ error of size at most $\exp(-O(\log^{1/3}(N)))$. The functions $\mu_{Siegel}$ and $\Lambda_{Siegel}$ can be decomposed into a convex combination of type I sums, twisted type I sums with $R = \exp(O(\log^{1/2}(N)))$, and type II sums up to an $L^1([N])$ error of size at most $\exp(-O(\log^{1/3}(N)))$.
\end{lemma}
This doesn't give a desirable twisted type I contribution because $R = \exp(O(\log^{1/2}(N)))$ is too large to work with in our case. We wish to prove this for some smaller value of $R_1 = \exp(O(\log^{1/C}(N)))$. However, if $D \ge R_1/2$, truncating $[N]$ so that it is $[\delta N, N]$, then we see the twisted type I case ends up as
$$\left|\sum_{d \in [D/2, D]} \sum_{w, w' \in [\delta N/(2D), N/D]} f(dw) \overline{f(dw')} \chi(w)\overline{\chi}(w')1_{[N']}(dw)1_{[N']}(dw')\right| \ge \frac{\delta^2 N^2}{D\log^{O(1)}(N)}.$$
Dyadically decomposing in $w$, we see that there exists $W$ with $\delta N/4 \le DW \le N$ such that
$$\left|\sum_{d \in [D/2, D]} \sum_{w, w' \in [W/2, W]} f(dw) \overline{f(dw')}\chi(w)\overline{\chi}(w')1_{[N']}(dw)1_{[N']}(dw')\right| \ge \frac{\delta^2N^2}{D \log^{O(1)}(N)}.$$
This situation resembles exactly the type II case now. Arguing as in the type II case (e.g., Fourier expanding $1_{[N'/d]}(w)$ and using Cauchy-Schwarz in $d$), we obtain a type II hypothesis
$$\left|\sum_{d, d' \in [D/2, D]} \sum_{w, w' \in [W/2, W]} f(dw) \overline{f(dw')f(d'w)}f(d'w')\right| \ge \frac{\delta^{O(1)}N^2}{\log^{O(1)}(N)}.$$
Thus, we have converted a type I hypothesis to a type II hypothesis.

\subsection{Minor Arc Estimates for Locally Quadratic Forms}
We now reduce to the case of $q_{Siegel}$ is small. We will need the following lemma:
\begin{lemma}
Let $\phi\colon B(S, 64\rho) \to \mathbb{R}/\mathbb{Z}$ be a locally quadratic form. Then 
$$1_B \phi(x) = \frac{1}{|B_\epsilon|^3}\mathbb{E}_{h_1, h_2, h_3 \in G} 1_B(x + h_1 + h_2 + h_3)1_B(x + h_1 + h_2)1_B(x + h_2 + h_3)1_B(x + h_1 + h_3)$$
$$1_{B_\epsilon}(x + h_1) 1_{B_\epsilon}(x + h_2)1_{B_\epsilon}(x + h_3) \prod_{\omega \in \{0, 1\}^3 \setminus \{0^3\}} e(\phi(x + h \cdot \omega)) + O_{L^1}(\epsilon |S|)$$
where $B = B(S, \rho)$ and $B_\epsilon := B(S, \epsilon' \rho)$ with $\epsilon' \in [\epsilon/2, \epsilon]$ is chosen so that $B_{\epsilon}$ is regular.
\end{lemma}
\begin{proof}
If $x + h_1 + h_2 + h_3 \in B$ and $x + \omega \cdot h \in B_\epsilon$ for $\omega \in \{0, 1\}^3 \setminus \{0^3, 1^3\}$, then $x \in B(S, \rho(1 + 6\epsilon))$. On the other hand, if $x \in B$ and $h_1, h_2, h_3$ are chosen so that $x + h_1, x + h_2, x + h_3 \in B_\epsilon$, then $x + h_i + h_j \in B(S, \rho(1 + 3\epsilon))$ and $x + h_1 + h_2 + h_3 \in B(S, \rho(1 + 6\epsilon))$. By regularity of $B$, and using that $\phi$ is quadratic, it follows that the two expressions differ by at most $O_{L^1}(\epsilon |S|)$. 
\end{proof}
If the conclusion of Section 4.1 holds, e.g.,
\begin{equation}
|\mathbb{E}_{n \in [N]} (\mu - \mu_{Siegel})(n) \psi(n) e(\phi(n))| \ge \delta^{O(|S|)^{O(1)}} \tag{3}
\end{equation}
or
\begin{equation}
|\mathbb{E}_{n \in [N]} (\Lambda - \Lambda_{Siegel})(n) \psi(n) e(\phi(n))| \ge \delta^{O(|S|)^{O(1)}} \tag{4}
\end{equation}
for $\psi$ a smooth approximant to $1_B$ supported in $B$. Then by the above lemma, \thref{TT1}, and \thref{CauchySchwarzGowers}, it follows that
\begin{align*}
|\mathbb{E}_{n \in [N]} \mu_{Siegel}(n) \psi(n) e(\phi(n))| &\ll \delta^{O(|S|)^{O(1)}} + \delta^{-O(|S|)^{O(1)}}\|\mu_{Siegel}\|_{U^3([N])} \\
&\ll \delta^{O(|S|)^{O(1)}} + \delta^{-O(|S|)^{O(1)}}q_{Siegel}^{-c}
\end{align*}
or
\begin{align*}
  |\mathbb{E}_{n \in [N]} (\Lambda_{Siegel} - \Lambda_Q)(n) \psi(n) e(\phi(n))| &\ll \delta^{O(|S|)^{O(1)}} + \delta^{-O(|S|)^{O(1)}}\|\Lambda_{Siegel} - \Lambda_Q\|_{U^3([N])} \\
&\ll \delta^{O(|S|)^{O(1)}} + \delta^{-O(|S|)^{O(1)}}q_{Siegel}^{-c}. 
\end{align*}
If $q_{Siegel} \ge \delta^{-O(|S|)^{O(1)}}$, then inserting in (3) and (4), we obtain
$$|\mathbb{E}_{n \in [N]} \mu(n)\psi(n) e(\phi(n))| \gg \delta^{O(|S|)^{O(1)}}$$
$$|\mathbb{E}_{n \in [N]} (\Lambda - \Lambda_Q)(n) \psi(n) e(\phi(n))| \gg \delta^{O(|S|)^{O(1)}}.$$
In this case, we do not have a twisted type I sum after we apply \thref{decomposition}, and the setting is significantly simpler. In particular, \thref{type1type2} applies but without any twisted type I hypothesis. Thus, we shall assume that $q_{Siegel} \le \delta^{-O(|S|)^{O(1)}}$. The below theorem takes the application of \thref{decomposition} to the conclusion of Section 4.1 and the above discussion as hypotheses and outputs a conclusion about the behavior of the locally quadratic form $\phi$ on $B(S, \rho)$, which we can then use to prove \thref{maintheorem}.
\begin{theorem}\thlabel{type1type2}
Let $\phi\colon B = B(S, 64\rho) \to \mathbb{R}/\mathbb{Z}$ be a locally quadratic form with $\rho \gg \delta^{C}$, $|S| \le \log^{C}(1/\delta)$ and let $\phi''$ be the associated bilinear form defined above in Section 2. Let $0 < R < N$. Suppose either
\begin{itemize}
\item[(1)] (Type I case) for more than $\delta^{O(1)}D$ many $d \in [D/2, D]$ with $D \ll \sqrt{R}$,
$$\left|\sum_{n \in [\delta N, N], n \equiv 0 \pmod{d}} \psi(n)e(\phi(n))\right| \ge \delta N$$ 
or
\item[(2)] (Twisted Type I case) for more than $\delta^{O(1)}D$ many $d \in [D/2, D]$ with $D \ll \sqrt{R}$,
$$\left|\sum_{n \in [\delta N, N], n \equiv 0 \pmod{d}} \psi(n) \chi_{Siegel}(n/d)e(\phi(n))\right| \ge \delta N$$
or
\item[(3)] (Type II case) there exists $L$ and $M$ with $\sqrt{R}\le L \le \frac{N}{\sqrt{R}}$ and $\delta N \le LM \le N$ such that
$$\left|\sum_{\ell, \ell' \in [L, 2L], m, m' \in [M, 2M]} \psi(m\ell) \psi(m'\ell) \psi(m\ell') \psi(m'\ell') e(\phi(m\ell) - \phi(m'\ell) - \phi(m\ell') + \phi(m'\ell'))\right| \ge \delta LM.$$
\end{itemize} 
Then at least one of the following holds:
\begin{itemize}
    \item[(1)] (Type I/twisted type I case) There exists $q \ll \delta^{-O(|S|)^{O(1)}}$ such that
    $$\|q\phi''(a, b)\|_{\mathbb{R}/\mathbb{Z}} \ll \delta^{-O(|S|)^{O(1)}} R^{O(1)}\frac{\|a\|_S\|b\|_S}{\rho_1^2}$$
    whenever $a, b \in B(S, \frac{\delta^{O(1)}\rho_1}{O(|S|)^{|S|}})$ with $\rho_1 = \rho/R^{O(1)}$.
    \item[(2)] (Type II case) there exists $q \ll \delta^{-O(|S|)^{O(1)}}$ such that
$$\|q \phi''(a, b)\|_{\mathbb{R}/\mathbb{Z}}  \ll \delta^{-O(|S|)^{O(1)}}\frac{\|a\|_{S}\|b\|_{S}}{\rho^2}$$
whenever $a, b \in B(S, \frac{\delta^{O(1)}\rho}{O(|S|)^{|S|}})$.
    \item[(3)] (Degenerate Case I) $R \ll \delta^{-O(|S|^{O(1)})}$
    \item[(4)] (Degenerate Case II) $\delta \ll \exp(-O(\log^{1/O(1)}(N)))$
    \item[(5)] (Degenerate Case III) $q_{Siegel} \ge \delta^{-O(|S|)^{O(1)}}$. If hypothesis (1) or (3) does occurs, then one of conclusions (1)-(4) also holds.
\end{itemize}

\end{theorem}
Assuming that the three degenerate conditions of \thref{type1type2} do not hold, we show \thref{maintheorem}. We shall take $R = \delta^{-O(|S|)^{O(1)}}$. Then the theorem above states that our quadratic phase $\phi$ whose quadratic part is roughly rational with denominator $q$ on a small Bohr set. By splitting $N$ into arithmetic progressions of common difference $q$, we may essentially assume that $\phi''$ is roughly constant on a small Bohr set. The point is that we will now be able to conclude that $\phi$ has Fourier complexity bounded by $\delta^{-O(|S|)^{O(1)}} (|S|)^{O(|S|)^{O(1)}}$ with the following lemma combined with \thref{locallylinear}. 
\begin{lemma}
Suppose $M$ is a bilinear form defined on $B(S, \rho)$ which varies by at most $\epsilon^{2|S|^3}$ on $B(S, \epsilon \rho)$. Then $e(M) 1_{B(S, \rho)}$ is an $(\epsilon \rho)^{O(|S|)^{O(1)}}$-approximate degree one nilsequence with $\delta$-Fourier complexity at most
$O((\delta\epsilon \rho |S|)^{-O(|S|^{O(1)})})$.
\end{lemma}
\begin{proof}
We may cover $B(S, \rho)$ with at most $\epsilon^{-2|S|^{2}}$ many translates of $B(S, \epsilon \rho)$, labeled $a_i + B(S, \epsilon \rho)$. On each translate $a_i + B(S, \epsilon \rho)$, we have $M(a_i + x, a_i + x) = M(x, x) + 2M(a_i, x) + M(a_i, a_i)$ with $x \in B(S, \epsilon \rho)$. Since $M(x, x)$ varies at most $\epsilon^{2|S|^3}$ on $B(S, \epsilon \rho)$, $M$ behaves like a linear function on $a_i + B(S, \epsilon \rho)$. Thus, taking a partition of unity $\eta_i$ with respect to $a_i + B(S, \epsilon \rho)$, we may write
$$e(M) \psi(S, \rho) = \sum_{q \in [Q]} \sum_i \eta_i e(M) = \sum_i \eta_i e(2M(a_i, x) + M(a_i, a_i)) + O(\epsilon^{|S|^3}).$$
The claim then follows from \thref{BohrFourier} and \thref{locallylinear}.
\end{proof}
Thus, we may essentially replace $1_B e(\phi)$ with a Fourier phase, leading to the conclusion that
$$\|\Lambda - \Lambda_{Siegel}\|_{U^2([N])} \gg \delta^{O(|S|)^{O(1)}}$$
$$\|\mu - \mu_{Siegel}\|_{U^2([N])} \gg \delta^{O(|S|)^{O(1)}}.$$
Finally, by \thref{TT1}, we see that
$$\|\Lambda - \Lambda_{Siegel}\|_{U^2([N])} \le \exp(-O(\log^{1/O(1)}(N)))$$
$$\|\mu - \mu_{Siegel}\|_{U^2([N])} \le \exp(-O(\log^{1/O(1)}(N)))$$
leading to the desired estimates that
$$\|\Lambda - \Lambda_{Siegel}\|_{U^3([N])} \le \exp(-O(\log^{1/O(1)}(N)))$$
$$\|\mu - \mu_{Siegel}\|_{U^3([N])} \le \exp(-O(\log^{1/O(1)}(N)))$$
concluding the proof of \thref{maintheorem}. Thus, we are reduced to proving \thref{type1type2}.
\section{Controlling the Type I and Twisted Type I Sum}
In this section, we will control the type I and twisted type I sums, ultimately turning the property that $e(\phi)$ satisfies some oscillatory inequality as in hypotheses (1) and (2) of \thref{type1type2} to a property that $\phi''$ is small on some subset of a Bohr set. We will use this information in Section 7 to obtain conclusion (1) of \thref{type1type2}. The main tool that allows us to exploit cancellation in the oscillatory sum is \thref{philemma} which offers a ``coordinate-free" method to obtain an equidistribution-like estimate for quadratic phases on Bohr sets rather than an alternate method of \cite{GW} which passes to a generalized arithmetic progression. Since generalized arithmetic progressions don't behave well with respect to multiplication, it seems necessary to to use a slightly different approach from \cite{GW}.
\begin{proposition}\thlabel{type1case}
Let $\phi\colon h + B = h + B(S, 64\rho) \to \mathbb{R}/\mathbb{Z}$ be a locally quadratic form with $\rho \gg \delta^{-C}$, $|S| \le \log^{C}(1/\delta)$ and $h \in [N]$. Let $\psi_h = \psi(\cdot - h)$ be a smooth approximant to $1_{B + h}$ with support in $B + h$. Suppose we have the estimate $q_{Siegel} \le \delta^{-O(|S|)^{O(1)}}$. Let $L \le \sqrt{R}$. Suppose $N' \ge \delta^2 N$ and that for $d \ll L$ that
$$\left|\sum_{n \in [N'/d]} \psi_h(dn) e(\phi(dn))\right| \ge \delta N/d$$
or
$$\left|\sum_{n \in [N'/d]} \psi_h(dn) \chi_{Siegel}(n)e(\phi(dn))\right| \ge \delta N/d.$$
Then for $a \in [N]$ chosen so that $ad \in B(S, \rho\delta^{O(|S|)^{O(1)}})$,
$$\|\phi''(ad, ad)\|_{Q, \mathbb{R}/\mathbb{Z}} \le \delta^{-O(|S|)^{O(1)}} \|ad\|_S^2 \rho^{-2}.$$
\end{proposition}
\begin{proof}
As noted before, we shall work with $B$ instead of $h + B$ since all of the estimates still hold for $h + B$. In particular, \thref{philemma} holds for $\phi$ defined on a translation $h + B$ where we note that its second derivative $\phi''$ is still defined on $B \times B$ (rather than $(h + B) \times (h + B)$). We start with the hypotheses
$$\left|\sum_{n \in [N]} \psi(dn)1_{[N'/d]}(n) e(\phi(dn))\right| \ge \delta N/d$$
and
$$\left|\sum_{n \in [N]} \psi(dn)1_{[N'/d]}(n) \chi_{Siegel}(n) e(\phi(dn))\right| \ge \delta N/d.$$ 
The second hypothesis is effectively the same as the first hypothesis because $\chi_{Siegel}$ is a major arc because it is periodic with respect to the conductor of the Siegel zero. Since we have reduced to the case of $q_{Siegel} \le \delta^{-O(|S|)^{O(1)}}$, we may safely handle the second case at the cost of replacing $\delta$ with $\delta^{O(|S|)^{O(1)}}$. Let us now consider only the type I case. Let $a \in [N]$ be such that $da \in B(S, \delta^{O(|S|)^{O(1)}}\rho/2)$. Let $A = \frac{\|da\|_S}{\rho}$. Then
$$\sum_{n \in [N]} \psi(dn)1_{[N'/d]}(n) e(\phi(dn)) = \sum_{n \in [N]} \mathbb{E}_{q \in [-A, A]} \psi(d(n + qa))1_{[N'/d]}(n + qa) e(\phi(d(n + aq))) \ge \delta N/d.$$
By the pigeonhole principle, we have for some $n$,
$$|\mathbb{E}_{q \in [-A, A]} \psi(d(n + qa)) 1_{[N'/d]}(n + qa) e(\phi(d(n + qa)))| \ge \delta.$$
The next lemma, which is \cite[Corollary 9.1]{GT4}, tells us how to utilize quadratic structure to obtain a quadratic polynomial in $q$.
\begin{lemma}\thlabel{philemma}
There exists $\alpha, \beta \in \mathbb{T}$ such that $\phi(dn + dqa) = \frac{q(q - 1)}{2} \phi''(ad, ad) + \alpha q + \beta$.
\end{lemma}
\begin{proof}
We shall prove this by induction on $q$. From the identity $\phi(n + a + b) - \phi(n + a) - \phi(n + b) + \phi(n) = \phi''(a, b)$ with $a = b = da$, we obtain
$$\phi(dn + d(q + 2)a) - 2\phi(dn + d(q + 1)a) + \phi(dn + dqa) = \phi''(da, da).$$
The claim of $q = 0, 1$ follows vacuously from the claim. By induction on $q$, there exists $\alpha, \beta \in \mathbb{T}$ such that
$$\phi(dn + d(q + 2)a) - q(q + 1) \phi''(da, da) - 2\alpha (q + 1) - 2\beta + \frac{q(q - 1)}{2} \phi''(da, da) + \alpha q + \beta = \phi''(da, da).$$
Rearranging, it follows that
$$\phi(dn + d(q + 2)a) = \frac{(q + 2)(q + 1)}{2} \phi''(ad, ad) + (q + 2)\alpha + \beta.$$
\end{proof}
Since by \thref{BohrFourier} and $1_{[N'/d]}$ has $\delta^4$-Fourier complexity at most $\delta^{-O(1)}$, there exists $\alpha$ such that
$$\left|\mathbb{E}_{q \in [-A, A]} e\left(\frac{1}{2}q(q - 1)\phi''(ad, ad) + \alpha q\right)\right| \ge \delta^{O(|S|)^{O(1)}} .$$
Thus, by \thref{multi}, it follows that
$$\|\phi''(ad, ad)\|_{Q, \mathbb{R}/\mathbb{Z}} \le \delta^{-O(|S|)^{O(1)}} A^{-2} = \delta^{-O(|S|)^{O(1)}} \|ad\|_S^2 \rho^{-2}.$$
\end{proof}

\section{Controlling the Type II sum}
In this section, we shall analyze the oscillatory type II as in hypothesis (3) of \thref{type1type2}, ultimately aiming to prove that $\phi''$ is small on a somewhat regular subset of $B(S, \rho)$. We will use this information to obtain conclusion (2) of \thref{type1type2}. \thref{quartic} is the main tool to analyze the equidistribution sum in the type II sum similar to how \thref{philemma} is the main tool to analyze the equidistribution sum of the type I sum. Since type II sums also involve multiplication in the domain of $\phi$, we also similarly cannot straightforwardly pass to a generalized arithmetic progression as in \cite{GW}.
\begin{proposition}\thlabel{type2case}
Let $L, M$ be chosen such that $\delta N \le LM \le N$, $\sqrt{R} \le L \le N/\sqrt{R}$ with $\delta > 0$. Let $\phi: h + B = h + B(S, 64\rho) \to \mathbb{R}/\mathbb{Z}$ be a locally quadratic form and $\psi_h = \psi(\cdot - h)$ be a smooth approximant to $1_{B + h}$ with support in $B + h$. Suppose that the type II sum holds:
$$|\mathbb{E}_{\ell, \ell' \in [L, 2L], m, m' \in [M, 2M]} \psi_h(m\ell) \psi_h(m'\ell) \psi_h(m\ell') \psi_h(m'\ell') e(\phi(m\ell) - \phi(m'\ell) - \phi(m\ell') + \phi(m'\ell'))| \ge \delta.$$
Then there exists an absolute constant $C$ such that for $\delta^{-Cr^{C}} \le K \ll R^{1/8}$, $a \in [L/2K, L/K]$ and $b \in \mathbb{Z}$ satisfying $\|ab\|_{S} \le \frac{1}{K^2}$ and $|ab|\le N$, we have
$$\|\phi''(ab, ab)\|_{Q, \mathbb{R}/\mathbb{Z}} \le \delta^{-O(|S|)^{O(1)}}/K^4$$
for $Q \le \delta^{-O(|S|)^{O(1)}}$.
\end{proposition}

\begin{proof}
As before, we will work with the constraint of $h = 0$ since all of the estimates of $h = 0$ work for in the more general case of $\phi$ being defined on a translation $h + B$; in particular, \thref{quartic} is proved in \cite[lemma 9.2]{GT4} for translations of Bohr sets rather than Bohr sets so it still holds in this case. We begin with the constraint
$$|\mathbb{E}_{\ell, \ell' \in [L, 2L], m, m' \in [M, 2M]} \psi(m\ell) \psi(m'\ell) \psi(m\ell') \psi(m'\ell') e(\phi(m\ell) - \phi(m'\ell) - \phi(m\ell') + \phi(m'\ell'))| \ge \delta$$
By pigeonholing in $\ell, m$ (in particular eliminating $\ell', m'$), we obtain
$$|\mathbb{E}_{\ell \in [L, 2L], m \in [M, 2M]} \psi(\ell m)e(\phi(m\ell))b(\ell)b(m)| \ge \delta$$
where in this proof, $b$ denotes a function whose absolute value is bounded by $1$. Let $a \in [L/K, L/2K]$ and $b \in \mathbb{Z}$ be such that $\|ab\|_S \le K^{-2}$ and $|ab| \le \frac{N}{2}$. Then since $|ab| \le N\|ab\|_S$ (since we specified in Section 2 that we work with the convention that $1 \in S$), it follows that $|b| \le \frac{N\|ab\|_S}{a} \le \frac{N}{KL} \le \delta^{-1} \frac{M}{K}$. Taking $U = V = \frac{\delta^2 K}{4}$, it follows that $U|a| \le \delta L$ and $V|b| \le \delta M$. Thus, $\|ab\|_S \le \frac{\delta^{-3}}{UV}$ with $U|a| \le \frac{\delta}{2} L$, $V|b| \le \frac{\delta}{2}M$. Our initial goal is to replace an average in $m\ell$ with an average in $[U]$ and $[V]$ over small multiples of $a$ and $b$. To see why that may be helpful, we have the following lemma from \cite[Lemma 9.2]{GT4}:
\begin{lemma}\thlabel{quartic}
Let $\partial_y(f) = f(x + y) - f(x)$ and define $f(v, w) = (m + va)(\ell + wb)$. Suppose
$$UV\|ab\|_S \le \rho$$
and for all $u_0, u_1, u_2 \in [U]$, $v_0, v_1, v_2 \in [V]$,
$$f(u_0 + \omega_{1, 1} u_1 + \omega_{1, 2}u_2, v_0 + \omega_{2, 1}v_1 + \omega_{2, 2} v_2) \in B(S, \rho/10)$$
for any choice of $\omega_{i, j} \in \{0, 1\}$. Then
$$\partial^1_{u_1, u_2}\partial^2_{v_1, v_2}\phi(f(u_0, v_0)) = 2u_1u_2v_1v_2 \phi''(ab, ab)$$
where $\partial^1$ applies in the first variable and $\partial^2$ applies in the second variable.
\end{lemma}
\begin{proof}
By replacing $m$ and $\ell$ by $m + u_0a$ and $\ell + v_0 b$, we may assume that $u_0, v_0$ are zero. Since $m\ell \in B(S, \rho)$ it follows that $u_ia \ell$ and $v_i bm$ are in $B(S, 2\rho)$. Thus by considering $n = \ell(m + \omega_{1, 1}au_1 + \omega_{1, 2} au_2)$, $h = (m + \omega_{1, 1} au_1 + \omega_{1, 2} a u_2) bv_1$ and $k = (m + \omega_{1, 1}au_1 + \omega_{1, 2} au_2) b v_2$, we have that
$$\partial_{v_1, v_2} \phi(f(\omega_{1, 1} u_1 + \omega_{1, 2} u_2, 0)) = \phi(n + h + k) - \phi(n + h) - \phi(n + k) + \phi(n).$$
Since $n, n + h_1, n + h_2, n + h_1 + h_2 \in B(S, 16\rho)$, this is equal to
$$\phi''(h_1, h_2) = \phi''((m + \omega_{1, 1} au_1 + \omega_{1, 2} a u_2) bv_1, (m + \omega_{1, 1} au_1 + \omega_{1, 2} a u_2) bv_2).$$
Using bilinearity, and the fact that $UV\|ab\|_S \le \rho$ and $u_i a\ell$ and $v_i bm$ are in $B(S, 2\rho)$ we can expand this in terms of
$$2\omega_{1, 1} \omega_{1, 2}u_0u_1v_0v_2 \phi''(ab, ab) + [\text{Lower Order Terms}]$$
where $[\text{lower order terms}]$ are degree one terms in $\omega_{1, 1}$ and $\omega_{1, 2}$. Taking an alternating sum in $\omega_{1, 1}$ and $\omega_{1, 2}$ which varies in $\{0, 1\}^2$, it follows that
$$\partial_{u_1, u_2}^1 \partial_{v_1, v_2}^2 \phi(f(u_0, v_1)) = 2u_1u_2v_1v_2 \phi''(ab, ab).$$
\end{proof}
This lemma tells us that it would be helpful if we were to convert an average in $L$ and $M$ into one along arithmetic progressions which are multiplies of $a$ and $b$ so that once we use Cauchy-Schwarz or a van der Corput type estimate sufficiently enough times, we obtain a genuine multidimensional polynomial that we may control. Now we average over $au$ and $bv$:
$$|\mathbb{E}_{\ell \in [L, 2L], m  \in [M, 2M]} \psi(\ell m)e(\phi(m\ell))b(\ell)b(m)| \le$$
$$|\mathbb{E}_{\ell \in [L, 2L], m \in [M, 2M]} \mathbb{E}_{u \in [U], v \in [V]} \psi((\ell + ua)(m + bv)) e(\phi((\ell + ua)(m + vb))) b(\ell + ua)b(m + vb)| + \frac{\delta}{2}$$
and thus there exists $\ell, m$ such that
$$|\mathbb{E}_{u \in [U], v \in [V]} \psi((\ell + ua)(m + vb)) e(\phi((\ell + ua)(m + vb))) b(u)b(v)| \ge \frac{\delta}{2}.$$
Applying Cauchy-Schwarz to eliminate $u$ and $v$, we have
$$|\mathbb{E}_{u, u' \in [P], v, v' \in [R]} \psi(h(u, v)) \psi(h(u', v)) \psi(h(u, v')) \psi(h(u', v'))$$
$$e(\phi(h(u, v)) - \phi(h(u', v)) - \phi(h(u, v')) + \phi(h(u', v')))| \ge \frac{\delta^4}{16}$$
where $h(u, v) = (\ell + ua)(m + vb)$. Making a change of variables $u' = u + u_0$ and $v' = v + v_0$, we have
$$|\mathbb{E}_{u \in [U], v \in [V], u_0 \in [-U, U], v_0 \in [-V, V]} \psi(h(u, v)) \psi(h(u + u_0, v)) \psi(h(u, v + v_0)) \psi(h(u + u_0, v + v_0))$$
$$e(\partial_{u_0, v_0}\phi(h(u, v)))| \ge \frac{\delta^4}{32}$$
where the average over $u_0, v_0$ is over a ``tent," meaning that it is multiplied by $\nu_U(u_0) = \max(0, 1 - |u_0|/|U|)$. Using Cauchy-Schwarz over $u, v$ again, we obtain
$$|\mathbb{E}_{u \in [U], v \in [V], u_0, u_1 \in [-U, U], v_0, v_1 \in [-V, V]} \Psi(u, v, u_0, v_0, u_1, v_1) e(\partial_{u_0, u_1, v_0, v_1}(\phi(h(u, v))))| \ge \frac{\delta^8}{1024}$$
where
$$\Psi(u, v, u_0, v_0, u_1, v_1) = \Delta_{(u_0, v_0), (u_1, v_1)} \psi(h(u, v))$$
with $\Delta_yf(x) := f(x + y)\overline{f(x)}$ and the average over $[-U, U]$ and $[-V, V]$ is a tent average. The point is that
$$\partial_{(u_0, v_0), (u_1, v_1)} \phi(u, v) = 2u_0v_0u_1v_1\phi''(ab, ab)$$
so applying \thref{BohrFourier} and since $\nu$ is a degree one nilsequence on a $1$-dimensional torus with Lipschitz norm $O(1)$, it follows that
$$|\mathbb{E}_{u, u' \in [-U, U], v, v' \in [-V, V]} e(2u_0v_0u_1v_1\phi''(ab, ab) + [\text{Lower Order Terms}])| \ge \delta^{O(|S|)^{O(1)}}$$
where $[\text{Lower Order Terms}]$ indicate degree three or lower multidimensional polynomials in $u_0, v_0, u_1, v_1$. By \thref{multi}, it follows that there exists $Q \le \delta^{-O(|S|)^{O(1)}}$ such that
$$\|\phi''(ab, ab)\|_{Q, \mathbb{R}/\mathbb{Z}} \le \frac{\delta^{-O(|S|)^{O(1)}}}{U^2V^2}$$
or
$$|U|, |V| \le \delta^{-O(|S|)^{O(1)}}.$$
This latter case can be ruled out by the hypothesis of our proposition by choosing $C$ appropriately. \\\\
Finally, the condition $|ab| \le \frac{N}{2}$ may be removed since $\|\phi''(-ab, -ab)\|_{\mathbb{R}/\mathbb{Z}} = \|\phi''(ab, ab)\|_{\mathbb{R}/\mathbb{Z}}$ whenever $ab$ lies in an appropriate Bohr set, and since one of the residue classes $ab$ or $-ab$ lies in $[0, N/2]$.
\end{proof}
\begin{remark}
We note that \thref{quartic} allows us to exploit cancellation in the $[L, 2L]$ \emph{and} $[M, 2M]$ sums above, in contrast to \cite{BL} which only exploit cancellation in one of the sums. This allows us to obtain a more powerful conclusion while dealing with a type II sum than \cite{BL}.
\end{remark}
\section{Reduction to the One-Step Case}
In this section, we shall use the conclusions of the previous two sections to prove \thref{type1type2}. In \thref{Bohr0} and \thref{Bohr1} we will show that $\|\phi''(a, a)\|_{Q, \mathbb{R}/\mathbb{Z}}$ is small for $a$ in some sufficiently large portion of a Bohr set $B$, with \thref{Bohr0} meant for the type I case and \thref{Bohr1} meant for the type II case. Using this, we will show in \thref{Bohr2} that $\|\phi''(a, a)\|_{Q^{O(1)}, \mathbb{R}/\mathbb{Z}}$ is small whenever $a$ is in some shrunken version of $B$. This allows us to use, via \thref{Bohr3} and \thref{Bohr4}, the equidistribution of multidimensional polynomials, where bounds are only single exponential in dimension. \\\\
Let $\delta = \exp(-O(\log^{1/C}(N)))$ be some parameter representing a quantity we would insert in \thref{type1type2} in order to avoid one of the degenerate cases, $|S| = \log^{O(1)}(1/\delta)$, and $\rho$ for now will equal $\delta^{O(1)}$, though we will end up replacing $\rho$ with smaller quantities shortly to fit our purposes. Let $B = B(S, \rho)$ be a Bohr set, $\phi$ a locally quadratic form on $B(S, 64\rho)$. \thref{type1case} and \thref{type2case} give us the following results:
\begin{itemize}
    \item (Type I case) for all $a \in E$ and $d \in \mathbb{Z}$ with $|E| \ge \delta^{O(1)} D$ and $\|ad\|_S \le \rho_1$, $E \subseteq [D/2, D]$, $Q \le \delta^{O(|S|)^{O(1)}}$
    $$\|\phi''(ad, ad)\|_{Q, \mathbb{R}/\mathbb{Z}} \le \delta^{-O(|S|)^{O(1)}} \|ad\|_S^2 \rho_1^{-2}.$$
    In this case, we shall take $\rho_1 = \rho/R^{O(1)}$. It is important to note that $E$ is not necessarily an interval. This is because if we were to treat the type I sum as a double sum in $a$ and $d$, we won't be able to obtain enough phase cancellation in the $d$-sum for the phase to exhibit one-step behavior over a large interval since the sum over $d$ is a very short interval.
    \item (Type II case) for all $a \in [L/2K, L/K]$ and $b \in \mathbb{Z}$ with $\|ab\|_S \le K^{-2}$, $Q \le \delta^{-O(|S|)^{O(1)}}$, and $K \le R^{1/8}$ but $K \ge \delta^{-O(|S|)^{O(1)}}$
    $$\|\phi''(ab, ab)\|_{Q, \mathbb{R}/\mathbb{Z}}\le \delta^{-O(|S|)^{O(1)}} K^{-4}.$$ 
\end{itemize}
We wish to obtain a condition similar to
$$\|\phi''(a, b)\|_{Q, \mathbb{R}/\mathbb{Z}} \le \epsilon^{-1} \|a\|_S\|b\|_S$$
for some appropriate $\epsilon > 0$ to complete the proof of \thref{type1type2}. Of course, $\epsilon$ may differ based on whether we are analyzing a type I/twisted type I or type II sum.
\begin{lemma}\thlabel{Bohr0}
Let $B = B(S, \rho)$, $B_b = \{n \in B: b | n\}$. Then $|B_b| \ge \frac{4^{-|S|}}{b}|B|$.
\end{lemma}
\begin{proof}
We may partition $B(S, \rho/2)$ as $B_{b, k} = \{n \in B(S, \rho/2): n \equiv k \pmod{b}\}$. By the pigeonhole principle, there exists some $k$ such that
$$|B_{b, k}| \ge \frac{|B(S, \rho/2)|}{b}.$$
Since $B_{b, k} - B_{b, k} \subseteq B_b$, it follows that
$$|B_b| \ge |B_{b, k}| \ge 4^{-|S|}|B(S, \rho)|/b.$$
\end{proof}
In the type I case, the above lemma immediately gives us
$$\|\phi''(a, a)\|_{Q, \mathbb{R}/\mathbb{Z}} \le \delta^{-O(|S|)^{O(1)}}\rho^2$$
whenever $a \in A$ with $|A| \ge \frac{4^{-|S|}}{R}|B(S, \rho)|$.
\begin{lemma}\thlabel{Bohr1}
Let $I \subseteq \mathbb{Z}_{> 0}$ be an interval and $\rho, \epsilon, \eta > 0$ and less than $1/2$, $B(S, \rho)$ be a Bohr set. Suppose for all $a, b \in [N]$ with $ab \in B(S, \rho)$ and $b \in I$ satisfies
$$\|\phi''(ab, ab)\|_{Q, \mathbb{R}/\mathbb{Z}} \le \epsilon^{-1}\rho^2$$
with $Q \le \eta^{-1}$. Then, assuming that $\rho$ is sufficiently small, there exists a set $A \subseteq B(S, \rho)$ such that for each $a\in A$
$$\|\phi''(a, a)\|_{Q, \mathbb{R}/\mathbb{Z}} \le \epsilon^{-1} \rho^2$$
and
$$|A| \ge |B(S, \rho)|16^{-|S|} \min_{b \in I} b^{-2}|I|^2 \log^{-5}(N).$$
\end{lemma}

\begin{proof}
Let
$$A = \bigcup_{b \in I} B_b.$$
By the hypothesis of the lemma, it follows that
$$\|\phi''(a, a)\|_{Q, \mathbb{R}/\mathbb{Z}} \le \epsilon^{-1} \rho^2$$
for all $a \in A$. It thus suffices to show that
$$|A| \ge |B(S, \rho)|16^{-|S|} \min_{b \in I} b^{-2}|I|^2 \log^{-5}(N).$$
To do so, we use the second moment method. Let $J$ be the set of primes in $I$. Let $m$ be the largest element of $I$. By Cauchy-Schwarz,
$$\frac{\left|\bigcup_{b \in J} B_b\right|}{|B|} \mathbb{E}_{n \in B} \left(\sum_{b \in J} 1_{B_b}(n)\right)^2 \ge \left(\mathbb{E}_{n \in B} \sum_{b \in J} 1_{B_b}(n)\right)^2.$$
We have
\begin{equation}
\mathbb{E}_{n \in B} \left(\sum_{b \in J} 1_{B_b}(n)\right)^2 \le \frac{\log^2(N)}{\log^2(2)}. \tag{5}
\end{equation}
Thus,
$$\frac{\left|\bigcup_{b \in J} B_b\right|}{|B|} \mathbb{E}_{n \in B} \left(\sum_{b \in J} 1_{B_b}(n)\right)^2 \le \frac{\left|\bigcup_{b \in J} B_b\right|}{|B|} \frac{\log^2(N)}{\log^2(2)}.$$
In addition, by \thref{Bohr0}, we have the lower bound:
$$\left(\mathbb{E}_{n \in B} \sum_{b \in J} 1_{B_b}(n)\right)^2 \ge \left(4^{-|S|} \sum_{b \in J} \frac{1}{b}\right)^2 \ge 16^{-|S|}|J|^2m^{-2}.$$
It follows that
$$\left|\bigcup_{b \in I} B_b\right| \ge |B|16^{-|S|}\log^{-5}(N)|I|^2m^{-2}$$
which is what we desired.
\end{proof}
\begin{remark}
Since we are estimating a lower bound on a set, it may seem unintuitive to make the set smaller by restricting the union to the primes of $I$. If we did not restrict to the primes of $I$ and instead opted to use the $O(|S|)^{\mathrm{th}}$ divisor function moment bounds as in \cite[Proposition 11.1]{GT4}, a bound such as (5) would not hold and methods of \cite[Proposition 11.1]{GT4} would give us
$$\left|\bigcup_{b \in I} B_b\right| \gg |B|\rho^{3/2} O(\log(N))^{-2^{O(|S|)}}$$
The problematic term is the last term, which is double exponential in $|S|$. If we were to stick with the low moment bounds of the divisor function, similar methods yield
$$\left|\bigcup_{b \in I} B_b\right| \gg |B|O(\rho)^{O(|S|)} O(\log(N))^{-O(1)}.$$
In this case, we don't have any terms double exponential in dimension, but $O(\rho)^{O(|S|)}$ is too small for \thref{Bohr2} to apply. Indeed, a key important feature of \thref{Bohr1} is that the lower bound on $\frac{|A|}{|B(S, \rho)|}$ does not depend on $\rho$.
\end{remark}

Since we will require $|S| = O(\log^{1/C_1}(N))$ for $C_1 > 0$, we may assume that $16^{-|S|}$ is much smaller than $\log^{-A}(N)$ for any power $A$. Thus, for our purposes, in the type II case, letting $I = [L/2K, L/K]$ and $\rho = K^{-2}$ we have the conclusion of the above lemma for $16^{-O(|S|)}|B(S, \rho)|$ many $a$. For the purposes of \thref{Bohr2}, we shall now replace $\rho$ with a smaller quantity so that it is much smaller than $16^{-O(|S|)}$, possibly imposing yet another condition of the form $R \gg O(\delta)^{-O(|S|)^{O(1)}}$ when we apply \thref{type2case}. We note that none of the lemmas and propositions in Sections 5, 6, and 7 require an upper bound on $R$ beyond the trivial bound of $R \le N$, so we never run into conflicting bounds on $R$ when running the argument. This allows great flexibility in our argument and is a principle reason why our argument carries through. \\\\
In the next lemma, we will use the fact that $\phi$ has one-step behavior on a large portion of the Bohr set and propagate it to obtain one-step behavior on the entire shrunken version of the Bohr set.


\begin{lemma}\thlabel{Bohr2}
Let $\rho, \delta, \epsilon, \eta > 0$, and suppose
$$\|\phi''(n, n)\|_{Q, \mathbb{R}/\mathbb{Z}} \le C\rho^2$$
for some $C \ge 1$, whenever $n \in A$ with $|A| \ge \epsilon|B(S, \rho)|$ with $Q \le \eta^{-1}$, $C \rho^2/100 \le \epsilon \eta$, and $|S|^{-1} \ge \epsilon$. Then for each $a \in B(S, \frac{\rho\epsilon^2}{100})$, we have
$$\|\phi''(a, a)\|_{Q, \mathbb{R}/\mathbb{Z}} \le C\rho^{-2} (\epsilon \eta)^{-O(1)}\|a\|_S^2.$$
\end{lemma}

\begin{proof}
Pick $a \in B(S, \frac{\rho \epsilon^2}{100})$ and $L = \lfloor \frac{\epsilon \rho}{\|a\|_S}\rfloor$.  Using regularity, we have
$$\mathbb{E}_{n \in B(S, \rho)} \mathbb{E}_{\ell \in [L]} 1_A(n + a\ell) \ge \epsilon/2$$
By the pigeonhole principle, it follows that there exists $n$ such that for more than $\epsilon L/2$ values of $\ell \in [L]$,
$$\|\phi''(n + a \ell, n + a\ell)\|_{Q, \mathbb{R}/\mathbb{Z}} \le C\rho^2.$$
Since $\phi''(n + a \ell, n + a\ell)$ is a degree two polynomial in $\ell$, it follows that there exists $q_\ell$ such that 
$$\|q_\ell (\ell^2 \phi''(a, a) + \text{[Lower Order Terms]})\|_{\mathbb{R}/\mathbb{Z}} \le C\rho^2$$
where $[\text{Lower Order Terms}]$ indicate polynomials in $\ell$ that are less than degree two. By the pigeonhole principle, there exists some $q \le \eta^{-1}$ such that for $\epsilon \eta L/2$ many $\ell$, we have
$$\|q(\ell^2 \phi''(a, a) + [\text{Lower Order Terms}])\|_{\mathbb{R}/\mathbb{Z}} \le C\rho^2.$$
Using Vinogradov's Lemma (e.g., \thref{QuadraticVinogradov}), either $\epsilon \eta < \frac{C\rho^2}{8}$ or $L \ll (\epsilon \eta)^{-O(1)}$ or
$$\|q \phi''(a, a)\|_{\mathbb{R}/\mathbb{Z}} \ll (\delta \eta)^{-O(1)}L^{-2}.$$
The first case is ruled out by the hypothesis of the lemma. If the second case occurs, then it follows that
$$\|a\|_S \gg \rho^2 (\epsilon \eta)^{O(1)}$$
which implies that
$$\|q \phi''(a, a)\|_{\mathbb{R}/\mathbb{Z}} \ll C(\delta \eta)^{-O(1)} \|a\|_S^2.$$
Both the second and third case yield a desired outcome.
\end{proof}

Applying \thref{Bohr2} with $A$ equal to the set in the conclusion of \thref{Bohr1}, $\eta = \delta^{O(|S|)^{C_1}}$, $\rho$ in the type II case and $\rho_1$ in the type I case, $\epsilon = 16^{-O(|S|)}$, we obtain that
$$\|\phi''(a, a)\|_{Q, \mathbb{R}/\mathbb{Z}} \le M\|a\|_{S}^2\rho^{-2}$$
for some $M = \delta^{-O(|S|)^{O(1)}}$ for the type II case and $M = (\delta/R)^{-O(|S|)^{O(1)}}$ for the type I case. By a polarization identity, we can establish the following lemma:
\begin{lemma}[Polarization]\thlabel{Bohr3}
Suppose $16  Q \delta^{-1} \rho^2 < \frac{1}{100}$ and
$$\|\phi''(a, a)\|_{Q, \mathbb{R}/\mathbb{Z}} \le \delta^{-1} \|a\|_{S}^2$$
whenever $a \in B(S, \rho)$. Then for each $a, b \in B(S, \rho)$,
$$\|\phi''(a, b)\|_{Q^2, \mathbb{R}/\mathbb{Z}} \ll Q^{O(1)}\delta^{-1}\|a\|_S\|b\|_S.$$
\end{lemma}
\begin{proof}
For $a, b \in B(S, \rho)$, suppose that $\|b\|_S \le \|a\|_S$, and pick $C$ to be the least integer such that $\|a\|_S \le C \|b\|_S$. Consider $\phi''(a + cb, a + cb) - \phi''(a - cb, a - cb)$. We have
$$\|\phi''(a + cb, a + cb) - \phi''(a - cb, a - cb)\|_{Q^2, \mathbb{R}/\mathbb{Z}} = $$
$$\| 4c^2\phi''(a, b)\|_{Q^2, \mathbb{R}/\mathbb{Z}} \le Q\delta^{-1} (\|a + cb\|^2_S + \|a - cb\|^2_S) \le 16 Q\delta^{-1} \|a\|_S^2$$
for all $c \in [C]$. By Vinogradov's Lemma (e.g., \thref{QuadraticVinogradov}), it follows that either
$$\|\phi''(a, b)\|_{Q^2, \mathbb{R}/\mathbb{Z}} \le Q^{O(1)}\delta^{-1} \|a\|_S^2/C^2 \ll Q^{O(1)}\delta^{-1} \|a\|_S\|b\|_S$$
or $C \ll 1$ in which case
$$\|\phi''(a, b)\|_{Q^2, \mathbb{R}/\mathbb{Z}} \ll Q \delta^{-1} \|a\|_S\|b\|_S.$$
Either case gives a desired outcome.
\end{proof}
Using \thref{Bohr3}, we are able to prove that whenever $a, b \in B(S, \rho)$, then
$$\|\phi''(a, b)\|_{Q^{O(1)}, \mathbb{R}/\mathbb{Z}} \le Q^{O(1)}M^{O(1)} \|a\|_{S} \|b\|_{S}.$$
Finally, using the following lemma, we are able to establish that $e(\phi'')$ is a degree one nilsequence.
\begin{lemma}[Clearing Denominators]\thlabel{Bohr4}
Let $\delta, \eta, \rho > 0$ and suppose
$$\|\phi''(a, b)\|_{Q, \mathbb{R}/\mathbb{Z}} \le \delta^{-1}\|a\|_S\|b\|_S\rho^{-2}$$
for some $Q \le \eta^{-1}$ and all $a, b \in B(S, \rho)$. Then there exists a common $q \le \eta^{-O(|S|)^{O(1)}}$ such that
$$\|q\phi''(a, b)\|_{\mathbb{R}/\mathbb{Z}} \le \delta^{-O(|S|)^{O(1)}}\|a\|_S\|b\|_S\rho^{-2}.$$
\end{lemma}

\begin{proof}
By \cite[Corollary 4.9]{GT5} (the geometry of numbers), we may find a $a_1, a_2, \dots, a_{|S|}$, $N_1, \dots N_{|S|}$ with
$$\prod_{i = 1}^{|S|} N_i = |S|^{O(|S|)} N$$
such that each $a, b \in B(S, \rho/|S|^{O(|S|)})$ can be written uniquely as
$$n = n_1 a_1 + n_2a_2 + \cdots + n_{|S|}a_{|S|}$$
$$m = m_1a_1 + m_2a_2 + \cdots + m_{|S|} a_{|S|}$$
with $n_i = |S|^{O(|S|)} N_i \|a\|_S$ and $\|a_i\|_S \le N_i^{-1}$. Thus, either $n_i = 0$ or $N_i \ge \rho^{-1} |S|^{-O(|S|)}$. Choose $q_{ij}$ with $\|q_{ij}\phi''(a_i, a_j)\|_{\mathbb{R}/\mathbb{Z}} \le \frac{\delta^{-1}}{N_iN_j \rho^2}$. Defining $q = \prod_{i, j} q_{ij}$, we see that
$$\|q \phi''(a, b)\|_{\mathbb{R}/\mathbb{Z}} \le |S|^{O(|S|)} \delta^{-O(|S|)^{O(1)}} \|a\|_S\|b\|_S\rho^{-2}.$$
\end{proof}

The point of this lemma is that we have a uniform $q$ that is only single exponential in $|S|$ for which all $\phi''(a, b)$ is approximately rational with denominator $q$. Applying \thref{Bohr4} to the type II case, we obtain
\begin{equation}
\|q\phi''(a, b)\|_{\mathbb{R}/\mathbb{Z}} \le K \|a\|_{S} \|b\|_{S} \tag{1}
\end{equation}
for $a, b \in B(S, \rho \delta^{O(|S|)^{O(1)}})$. Thus, if $\delta = \exp(-\log^{1/C_0}(N))$, then $K = \exp(\log^{c_1/C_0}(N))$, $q = \exp(\log^{c_2/C_0}(N))$ for some bounded (deterministic) constants $c_1$, $c_2$, $c_3$. Essentially, all parameters related to $\delta$, $\rho$, $K$, $M$ are quasi-polynomial in $N$, while $|S|$ is logarithmic in $N$. \\\\
In the type I case, we have (1) with $K = (\delta/R)^{-O(|S|)^{O(1)}}$ for $a, b \in B(S, \rho_1 \delta^{O(|S|)^{O(1)}})$. This completes the proof of \thref{type1type2}.


\section{Applications}
Let $\Psi = (\psi_1, \psi_2, \dots, \psi_k)\colon \mathbb{Z}^d \to \mathbb{Z}^k$ be affine forms such that no two differ by a constant. We may write
$$\psi_i = \dot{\psi_i} + \psi_i(0).$$
Define
$$\|\Psi\|_N := \sum_{i = 1}^k \sum_{j = 1}^d |\dot{\psi_i}(e_j)| + \sum_{i = 1}^k \left|\frac{\psi_i(0)}{N}\right|$$
where $e_1, \dots, e_d$ is the standard basis on $\mathbb{Z}^d$. We shall aim to improve error terms for a linear equations in primes type result for complexity two $\Psi$. The theory developed in \cite{GT1} is sufficient enough to give us a result for $\Psi$ having \textit{Cauchy-Schwarz complexity} at most two, i.e., for each $i$, one can partition $[k] \setminus \{i\}$ into three sets $J_1, J_2, J_3$ such that $\psi_i \not\in \text{span}\{\psi_k: k \in J_j\}$ for each $j = 1, 2, 3$. Specifically, if we define
$$\Lambda_{\Psi, N}(f_i) = \mathbb{E}_{v \in [N]^d} \prod_{i = 1}^k f_i(\psi_i(v))$$
then for $f_j$ all bounded by $L$,
$$|\Lambda_{\Psi, N}(f_j)| \ll L^{O(1)}\min_i \|f_i\|_{U^3([N])}^{c}$$
However, we can do better. Manners \cite{M2} recently proved that for $\Psi$ having \textit{true complexity} at most two at each $\psi_i$, i.e., for each $i$, $\psi_i^{\otimes 2}$ lies in the span of $\text{span}\{\psi_j^{\otimes 2}, j \neq i\}$ but $\psi_i^{\otimes 3} \not\in \text{span}\{\psi_j^{\otimes 3}\}$, then one can show, by first extending to a cyclic group of prime order and extending $f_j$ so that it is zero outside $[N]$, using only Cauchy-Schwarz that if $f_j$ were all bounded by $1$, then
$$|\Lambda_{\Psi, N}(f_j)| \ll \min_i \|f_i\|_{U^3([N])}^c.$$
Since the proof uses only finitely many Cauchy-Schwarz steps (which one can quantify based on $\|\Psi\|_N$), one obtains
\begin{equation}
|\Lambda_{\Psi, N}(f_j)| \ll L^{O_{\|\Psi\|_N, k}(1)} \min_i \|f_i\|_{U^3([N])}^{c_{\|\Psi\|_N, k}} \tag{6}
\end{equation}
whenever $f_j$ is $L$-bounded. \\\\
Let $K \subseteq [-N', N']^d$ be a convex body with $10N' \le N$ (we note that we are breaking our convention we set in Section 2 by requiring $[-N', N']$ to be the continuous interval rather than the discrete interval). Let
$$\Lambda_{\Psi, N, K}(f_j) = \mathbb{E}_{v \in [N]^d \cap K} \prod_{i = 1}^k f_i(\psi_i(v)).$$
Define the local factor
$$\beta_q = \Lambda_{\Psi, N}((\Lambda_{\mathbb{Z}/q\mathbb{Z}})_{i = 1}^k)$$
where $\Lambda_{\mathbb{Z}/q\mathbb{Z}}(x) = \frac{q}{\phi(q)}$ if $(x, q) = 1$ and is $0$ otherwise. Let
$$\beta_\infty = \text{vol}(K \cap \Psi^{-1}(\mathbb{R}^k_+)).$$
\begin{theorem}
Let $\Psi$ be affine forms such that no two $\psi_i$ differ by a constant. Suppose $\Psi$ has true complexity at most two at each $\psi_i$, the forms of $\Psi$ are not linearly (rationally) independent, and $\|\Psi\|_N \le L$. Then
$$\Lambda_{\Psi, N, K}((\Lambda)_{j = 1}^k) = \beta_\infty \prod_{p \text{ prime}} \beta_p + O_{L, d, k, A}(\log^{-A}(N))$$
where the constant is ineffective in $A$.
\end{theorem}
\begin{proof}
We first make a reduction to eliminate $K$. To do so, we apply \thref{fourierconvex} and \thref{FourierExpansion} to Fourier expand
$$1_K(n) = \sum_i a_i e(\alpha_i \cdot n) + O(\log^{-A}(N))$$
where $\sum_i |a_i| \ll \log^{O_{k, A}(1)}(N)$. If $N(\alpha_i \cdot n)$ lies in the span of $(\psi_j)_{j = 1}^k$, then we may write $N \alpha_i \cdot n$ into a linear combination of $\psi_j$ so $e(\alpha_i \cdot n)$ gets factored into a multiplicative combination of $e(\psi_j(n)/N)$. These terms may be absorbed in $f_j(\psi_j(n))$. They don't contribute to the $U^2$ or $U^3$ Gowers norm of $f_j$ because of part 2 of \thref{gowersproperties}. If $N(\alpha_i \cdot n)$ does not lie in the span of $(\psi_j)_{j = 1}^k$, then by orthogonality,
$$\mathbb{E}_{n \in [N]} e(\alpha_i \cdot n) \prod_{j = 1}^k f_j(\psi_j(n)) = 0.$$
Thus, by (6) and by part 2 of \thref{gowersproperties}, it follows that
\begin{equation}
|\Lambda_{\Psi, N, K} (\Lambda)_{j = 1}^k| \ll \min_{j} \log^{O_{k, A}(1)}(N)\|f_j\|_{U^3([N])}^{c_{L}} + O(\log^{-A}(N)). \tag{7}
\end{equation}
Letting $\Lambda_{Siegel}$ be the approximant of $\Lambda$ we defined, we may write
$$\Lambda_{\Psi, N, K}((\Lambda)_{j = 1}^k) - \Lambda_{\Psi, N, K}((\Lambda_{Siegel})_{j = 1}^k)$$
as a sum of $\Lambda_{\Psi, N, K}((g_j)_{j = 1}^k)$ where at least one of $g_j$ is equal to $\Lambda - \Lambda_{Siegel}$ and the rest of the terms are $O(\log(N))$-bounded. By \thref{maintheorem}, and (7), it follows that each of the
$$|\Lambda_{\Psi, N, K}((g_j)_{j = 1}^k)| \ll_{A, k, L} \log^{-A}(N).$$
We may also write $\Lambda_{Siegel} = \Lambda_Q + n^{\beta - 1}\frac{\phi(P(Q))}{P(Q)} 1_{(n, P(Q)) = 1}$. In addition, we may separate 
$$\Lambda_{\Psi, N, K}((\Lambda_{Siegel})_{j = 1}^k) - \Lambda_{\Psi, N, K}((\Lambda_Q)_{j = 1}^k)$$
into $O_k(1)$ terms of the form $\Lambda_{\Psi, N, K}((h_j)_{j = 1}^k)$, each of which has at least one $h_j$ that is of the form $\Lambda_{Siegel} - \Lambda_Q$ with the rest of the terms $O(\log(N))$-bounded. Then \thref{TT1} and (6) gives us that each
$$|\Lambda_{\Psi, N, K}((h_j)_{j = 1}^k)| \ll_{A, k, L} \log^{-A}(N).$$
Finally, \cite[Proposition 5.2]{TT} gives us
$$\Lambda_{\Psi, N, K}((\Lambda_Q)_{j = 1}^k) = \beta_\infty \prod_{p \text{ prime}} \beta_p + O_{L, d, k}(\exp(-O(\log^{1/O(1)}(N)))).$$
Putting everything together gives us the desired result.
\end{proof}
\subsection{Three Term Progressions with Shifted Prime Difference}
Let $w = \log^\epsilon(N)$ and $W = P(w)$. For a function $f$, we define the $W$-tricked $f$ at modulus $b$, denoted $f_{W, b}(n) := \frac{\phi(W)}{W}f(Wn + b)$. In this section, we will assume that Siegel zeros don't exist, so $\mu_{Siegel} = 0$ and $\Lambda_{Siegel} = \Lambda_{Q}$. It's possible that more effort can make results in this section unconditional, but we have chosen not to do so here.
\begin{lemma}\thlabel{wtrick}
If $\epsilon$ is sufficiently small,
$$\|(\Lambda)_{W, 1} - 1\|_{U^3([N/W])} \ll w^{-c}.$$
\end{lemma}
\begin{proof}
By \cite[Proposition 5.3]{TT}, we see that
$$\|(\Lambda_{Q})_{W, 1} - 1\|_{U^3([N/W])} \ll w^{-c}$$
so it suffices to show that
$$\|(\Lambda - \Lambda_Q)_{W, 1}\|_{U^3([N/W])} \ll w^{-c}.$$
We will actually show that
$$\|(\Lambda - \Lambda_Q)_{W, 1}\|_{U^3([N/W])} \ll W\exp(-O(\log^{1/C}(N))).$$
We Fourier expand $1_{n \equiv 1 \pmod{W}} = \mathbb{E}_{a \in [W]} e\left(\frac{a(n - b)}{W}\right)$, and using the triangle inequality and part 2 of \thref{gowersproperties}. We thus obtain
$$\|(\Lambda - \Lambda_Q)_{W, 1}\|_{U^3([N/W])} \ll  W\|\Lambda - \Lambda_Q\|_{U^3([N])} \ll W\exp(-O(\log^{1/C}(N))).$$
Thus, if $\epsilon$ is sufficiently small, we have that 
$$\|(\Lambda)_{W, 1} - 1\|_{U^3([N/W])} \ll w^{-c}.$$
\end{proof}
Using this, we show the following:
\begin{theorem}
Assume $\beta$ (the Siegel zero) doesn't exist. Suppose a subset $A$ of $[N]$ does not have any $3$ term arithmetic progressions with shifted prime common difference. Then there exists $c > 0$ such that
$$|A| \ll N\exp(-O(\log\log(N)^{c})).$$
\end{theorem}
\begin{proof}
We will follow the proof of \cite[Theorem 1.8]{TT}. Let $A$ be a subset of $[N]$ with density $\delta$. Let $\Lambda'$ be the restriction of $\Lambda$ to the primes. We wish to estimate
$$\frac{1}{\log(N)}\mathbb{E}_{n, d \in [N]} 1_A(n)1_A(n + d)1_A(n + 2d) \Lambda'(d + 1).$$
We will apply the $W$ trick to $n$ and $d$ in order for us to apply \thref{wtrick}. Applying the pigeonhole principle, there exists some modulus $b$ such that $A' := \{n: Wn + b \in A\}$ has size at least $\frac{\delta N}{W}$. Restricting to the progression $W \cdot [\cdot] + b$, we see that the number of three term progressions with shifted prime common difference is at least
$$\frac{1}{\log(N)} \sum_{n, d \le N/W} 1_{A'}(n) 1_{A'}(n + d) 1_{A'}(n + 2d) \Lambda'(Wd + 1).$$
Since $\|\Lambda - \Lambda'\|_{L^1} \ll N^{-1/2 + o(1)}$ it follows that
$$\|(\Lambda')_{W, 1} - 1\|_{U^3([N/W])} = \|(\Lambda)_{W, 1} - 1\|_{U^3([N/W])} + O(N^{-1/2 + o(1)}) \ll w^{-c}.$$
Thus, using the generalized von Neumann theorem as in \cite[Lemma 5.2]{TT2}, we can estimate the above sum as
$$\frac{W}{\phi(W)\log(N)} \left(\sum_{n, d \in [N/W]} 1_{A'}(n) 1_{A'}(n + d) 1_{A'}(n + 2d) + O((N/W)^2w^{-c})\right).$$
By Varnavides' trick \cite{Va}, we can estimate a lower bound on the density of three term arithmetic progressions which we denote $c(3, \delta)$ based on $r_3(N)$, the least density a set must be before it contains a non-trivial three-term arithmetic progressions as follows: letting $N_3(\alpha)$ denote the smallest positive integer such that for any $m \ge N_3(\alpha)$ any subset of $[m]$ of size at least $\alpha m$ contains a nontrivial three term arithmetic progression. One can take $N_3(\alpha)$ to be the smallest integer $N$ such that $r_3(N) \le \alpha N$. Then by a dilation argument due to Varnavides in \cite{Va}, we can take $c(3, \alpha) \ge \frac{\alpha^2}{16N_3(\alpha/2)^3}$. It follows that the number of three term arithmetic progressions with shifted primes is at least (for $N/W > 2N_3(\delta/2)$)
$$\frac{W}{\phi(W)\log(N)}\left(c(3, \delta)\left(\frac{N}{W}\right)^2 + O\left(\frac{N^2}{W^2 w^{c}}\right)\right).$$
If $c(3, \delta)\gg w^{-c}$, then $A$ has a three term arithmetic progression with shifted prime difference. By \cite{KM}, we can take $N_3(\delta/2) \ll \exp(O(\log^{O(1)}(1/\delta)))$ so 
$$c(3, \delta) \gg \exp(-O(\log^{O(1)}(1/\delta))).$$ 
This gives us that if $A$ has no three term arithmetic progressions of shifted primes, then $A$ has density at most $\exp(-O(\log\log(N)^{c}))$. 
\end{proof}
\begin{remark}
Unconditionally, using the above method, one can show $N\exp(-O(\log\log\log^{c}(N)))$ as done by \cite{TT}. This is because of a technical issue with the $W$ trick that the author encountered while trying to improve \cite[Theorem 2.5]{TT}: since
$$\|(\Lambda_{Siegel} - \Lambda_Q)_{W, 1}\|_{U^3([N/W])} = \|(\Lambda_{Q} (\cdot)^{\beta - 1} \chi_{Siegel})_{W, 1}\|_{U^3([N/W])}$$
it follows that if $W$ is too large, it's possible that $(\chi_{Siegel})_{W, 1}$ is constant leading to an estimate of $W^{c(\beta - 1)}$ of the above expression, which does not seem useful since Siegel's theorem is ineffective.
\end{remark}
It seems in view of \thref{maintheorem}, \cite{KM}, \cite{RS}, and \cite{W} that one can unconditionally hope for quasi-polynomial bounds for sets which lack three term arithmetic progressions of shifted prime common difference. The lack of such bounds in this section seems more of a limitation of the $W$ trick than anything else and it seems possible that other methods such as those of \cite{RS} and \cite{W} could yield significantly better bounds.

\appendix

\section{Sanders' improvement of the $U^3$ inverse theorem for $\mathbb{Z}/N\mathbb{Z}$}
In this section, we shall provide an argument for the following:
\begin{theorem}[Sanders \cite{S}, Green-Tao \cite{GT6}]\thlabel{sandersinversetheorem}
Let $f \colon \mathbb{Z}/N\mathbb{Z} \to \mathbb{C}$ be a function with $\|f\|_{L^{1024}(\mathbb{Z}/N\mathbb{Z})} \le L$, and $0 < \delta \ll 1$. Suppose $\|f\|_{U^3(\mathbb{Z}/N\mathbb{Z})} \ge \delta$. Then there exists a Bohr set $B(S, \rho)$, a locally quadratic form $\phi\colon B(S, 1000\rho) \to \mathbb{R}/\mathbb{Z}$ such that
$$\mathbb{E}_{x} |\mathbb{E}_{h \in B(S, \rho)} f(x + h) e(-\phi(h) - \xi(x) \cdot h)| \gg (\delta/L)^C$$
such that the Bohr set satisfies
$$|S| \le \log(L/\delta)^{C'}, \rho \ge \left(\frac{\delta}{LD}\right)^{C''}.$$
Here, $C, C', C'', D$ are absolute constants.
\end{theorem}
For completeness sake, we shall provide a fleshed out version of the argument suggested in \cite{GT6}, which relies on the following tool:
\begin{lemma}[Bogolyubov-Ruzsa lemma \cite{S}]\thlabel{BogolyubovRuzsa}
Suppose $G$ be a discrete abelian group and $A, S \subseteq G$ be subsets of $G$ such that
$$|A + S| \le K \min(|A|, |S|).$$
Then $(A - A) + (S - S)$ contains a proper symmetric $d(K)$-dimensional progression with size $\exp(-h(K))|A + S|$ where
$$d(K) = O(\log^{6}(2K)), \hspace{0.1in} h(K) = O(\log^6(2K)\log(2\log(K))).$$
\end{lemma}
\begin{lemma}[Bogolyubov-Ruzsa lemma, good model case]\thlabel{BogolyubovRuzsaGoodModel}
Let $A$ be a subset of a finite additive group $G$ such that $|A| \ge \delta |G|$. Then there exists a symmetric proper progression $P$ of dimension $d$ at most $O(\log(1/\delta))^{6 + o(1)}$ and size at least $\exp(-O(\log(1/\delta))^{6 + o(1)})|2A|$ such that $P \subseteq 2A - 2A$.
\end{lemma}
\begin{proof}
Note that $|A + A| \le |G| \le \frac{1}{\delta}|A|$. Then we apply \thref{BogolyubovRuzsa}.
\end{proof}
The above lemma will be used for the proof of the linearization lemma. The below lemma, a local version of the above lemma, will be used in the proof of the symmetry lemma.
\begin{lemma}[Improved Local Bogolyubov lemma]\thlabel{localbogolyubov}
Let $A$ be a subset of a Bohr set $B$ of dimension $d$ such that $|A| \ge \delta |B|$. Then there exists a symmetric proper progression $P'$ such that $P' \subseteq 2A - 2A$, $P'$ has dimension $O(\log^{6 + o(1)}(4^d/\delta))$ and size at least 
$$\exp(-O(\log(4^d/\delta))^{6 + o(1)})|B|.$$
\end{lemma}
\begin{proof}
Note that $|A + A| \le |2B| \le 4^{d} |B| \le \frac{4^d}{\delta}|A|$. Then apply \thref{BogolyubovRuzsa}.
\end{proof}
Though we can very well work with generalized arithmetic progressions (using what Gowers and Wolf \cite{GW} refer to as \emph{Bourgain systems}), we shall stay faithful to Green and Tao's presentation using Bohr sets. We will need the following lemma which appears in \cite[Proposition 27]{Mi}:
\begin{lemma}[symmetric coset progressions contain Bohr sets, \cite{Mi}]\thlabel{symmetricBohr}
Let $P$ be a symmetric coset progression of dimension $d$ and density $\alpha$. Then $P$ contains a Bohr set $B(S, \rho)$ with
$$|S| \le (d\log(\alpha^{-1}))^{C}$$
$$\rho \ge (d\log(\alpha^{-1}))^{-C'}$$
where $C$ and $C'$ are absolute constants.
\end{lemma}
\begin{remark}
Note that classical additive energy methods (e.g., \cite[Theorem 4.42]{TV}) do not seem to give good enough bounds for us to use. However, taking repeated sumsets as in \cite[Proposition 27]{Mi} makes the Fourier coefficients decrease exponentially with respect to the number of sumsets being taken, and from that one gets genuine improvement over \cite[Theorem 4.42]{TV}.
\end{remark}

It turns out that inserting these lemmas into the proof of the $U^3$ inverse theorem gives an improved version of a linearization lemma, which is primarily where the quantitative improvement comes from:
\begin{lemma}[Linearization lemma]\thlabel{linearization}
Let $G$ be a finite abelian group, $0 < \delta \le \frac{1}{10000}$, $f\colon G \to \mathbb{C}$ a function with $\|f\|_{L^{1024}[N]} \le L$ with $\|f\|_{U^3} \ge \delta$. Then there exists a regular Bohr set $B(S, \rho)$ with $|S| \ll \log^{C_1}(L\delta^{-1})$ and $\rho \ge O(\log^{-C_2}(L/\delta))$ and a locally linear function $M\colon B(S, 2000\rho) \to \hat{G}$ such that
$$|\mathbb{E}_{x \in G} \mathbb{E}_{h \in B(S, \rho)} b_2(h)b_2(x) f(x + h) e(-2Mh \cdot x)| \ge (\delta/L)^{O(1)}$$
for constants $C_1, C_2 > 0$ where $b_2$ consists of functions of the form $f(x + k)f(x)e(\phi(x, k))$, correlations of at most two terms $f$ and possibly a phase function $\phi$.
\end{lemma}
Note that $b_2$ is reminiscent of the one-bounded function convention common in related literature (e.g., \cite{GT7}). In this appendix, $b_\ell$ will denote a one-bounded function times $\ell$ shifts of $f$ or $\overline{f}$. For example, $b_3$ can denote $f(x + k_1)f(x + k_2)\overline{f(x + k_3)} e(\phi)$ where $\phi$ is some phase. The next step after the linearization lemma is the symmetry argument, which roughly states that on a smaller Bohr set $B_3$, $2Mh \cdot x \approx 2Mx \cdot h \pmod{1}$.
Following the argument of Green-Tao, this yields a quasi-polynomial $U^3$ inverse theorem:
\begin{theorem}
Let $f\colon \mathbb{Z}/N\mathbb{Z} \to \mathbb{C}$ be a function with $\|f\|_{L^{1024}[N]} \le 1$ and $0 < \delta < \frac{1}{100}$. Suppose
$$\|f\|_{U^3} \ge \delta.$$
Then there exists some absolute constant $C$, a degree two nilsequence $F(g(n)\Gamma)$ of complexity $\exp(O(\log^{C}(1/\delta)))$ on a nilmanifold of dimension $d \le O(\log^C(1/\delta))$ such that
$$|\langle f, F(g(\cdot)) \rangle| \ge \exp(-O(\log^C(1/\delta))).$$
\end{theorem}
\begin{remark}
One should in theory be able to prove a variant of the above $U^3$ inverse theorem with $L^p$ bound hypothesis for any $p$ strictly larger than $2$ using a weak regularity approach, but we shall not do so here. See e.g., \cite[Section 6]{KMT} the discussion on Hahn-Banach theorem.
\end{remark}
Now we shall perform a quantitative analysis of these lemmas followed by an actual proof for completeness sake. 
\subsection{A Quantitative Analysis}
In this subsection, we perform a quantitative analysis of the argument below, mostly keeping track of the parameters related to Bohr sets in the argument. As in the Green-Tao argument, we shall need to keep track of Bohr sets $B_1, B_2, B_3, B_4, B_5$ in our argument below for the $U^3$ inverse theorem. Noting that the quantitative loss occurs in the application of the inefficient Bogolyubov type lemmas we use the more efficient \thref{BogolyubovRuzsaGoodModel} in its place. Write $B_i = B(S_i, \rho_i)$. We see that
$$|S_1| \le  O(\log^{O(1)}(L/\delta)), \hspace{0.1in} \rho_1 \ge O(\log^{-O(1)}(L/\delta))$$
$$|S_2| = |S_1|, \hspace{0.1in} \rho_2 \ge (\delta/L)^{O(1)}.$$
Computing $|S_3|$ and $\rho_3$ requires the use of \thref{localbogolyubov} giving that $2A - 2A$ contains a generalized arithmetic progression of dimension $O(|S_2| + \log(L/\delta))^{6 + o(1)}$ and size at least $\exp(-O(|S_2| + \log(L/\delta))^{6 + o(1)})|B_2|.$ Applying \thref{symmetricBohr} and passing from $B_3$ to $\{2x: x \in B_3\}$, we obtain that 
$$|S_3| = O(|S_2| (|S_2| + \log(L/\delta) + \log(|B_2|)))^{O(1)}, \hspace{0.1in} \rho_3^{-1} \le O(|S_2|(|S_2| + \log(L/\delta) + \log(|B_2|)))^{O(1)}.$$
Plugging in $|S_2| \le O(\log^{O(1)}(L/\delta))$, $\rho_2 \ge (\delta/L)^{O(1)}$, we obtain $|B_2| \ge (\delta/L)^{O(\log(L/\delta))^{O(1)}} |G|$
$$|S_3| \le O(\log(L/\delta))^{O(1)}, \hspace{0.1in} \rho_3 \ge O(\log(L/\delta))^{-O(1)}.$$
Finally, an inspection of the rest of the argument gives
$$|S_4| = |S_3|, \hspace{0.1in} \rho_4 = (\delta/L)^{O(1)}$$
$$|S_5| = |S_4|, \hspace{0.1in} \rho_5 \ge (\delta/L)^{O(1)} \rho_4$$
where $\rho_5$ is chosen so that $B_5$ is regular.
\subsection{Proof of the Linearization lemma}
We will need the following from \cite[Proposition 5.1]{GT7}
\begin{lemma}
Let $f\colon [N] \to \mathbb{C}$ and $\delta < 1000^{-1}$ be a function with $\|f\|_{L^{1024}[N]} \le L$ and suppose $\|f\|_{U^3} \ge \delta$. Then there exists a set $H \subseteq [N]$ and a phase function $\xi\colon H \to \hat{G}$ such that
$$|\{h_1 + h_2 = h_3 + h_4\}| \ge N^3 (\delta/L)^{O(1)}$$
and
$$|\mathbb{E}_{x} f(x + h)\overline{f(x)} e(-\xi(h) \cdot x)| \ge (\delta/L)^{100}.$$
\end{lemma}
\begin{proof}
The proof essentially follows the proof of \cite[Proposition 5.1]{GT7} line-by-line, except we keep track of the Cauchy-Schwarz terms we apply.
\end{proof}
Following the argument in \cite[Section 5]{GT7}, we see that there exists some set $H$ and a function $\xi: H \to \hat{G}$ such that $\Gamma = \{(h, \xi(h))\}$ has size at least $(\delta/L)^{O(1)}|G|$ with
$$|kH - \ell H| \le ((\delta/L)^{-O(1)})^{k + \ell} |H|$$
such that for each $(h, \xi(h)) \in \Gamma$, we have
$$|\mathbb{E}_{x \in G} f(x + h) \overline{f(x)} e(-\xi(h) x)| \ge (\delta/L)^{O(1)}.$$
\begin{lemma}
There exists a subset $\Gamma' = \{(h, \xi(h)): h \in H'\}$ such that $|\Gamma'| \ge (\delta/L)^{O(1)}|G|$ and that $4\Gamma' - 4\Gamma'$ is a graph.
\end{lemma}
\begin{proof}
See \cite[lemma 9.2]{GT7}.
\end{proof}
Applying \thref{BogolyubovRuzsaGoodModel} and \thref{symmetricBohr} to $H'$, we find a regular Bohr set $B_1 = B(S, \rho)$ with $|S| \le \log^{O(1)}(L/\delta)$ and $\rho \ge \log^{-O(1)}(L/\delta)$ that is completely contained in $2H' - 2H'$. Thus, there exists a a function $2M\colon B(S, \rho) \to \hat{G}$ which is a Freiman homomorphism since $4\Gamma' - 4\Gamma'$ is a graph. Since $4\Gamma' - 4\Gamma'$ contains $0$, it follows that $M(0) = 0$. As $2M(h)$ is not quite contained in $H'$, but is contained in $2H' - 2H'$, we must, however, isolate a subset which is a translation of $H'$ for which $B(S, \rho)$ is dense in. \\\\
By the Plunnecke-Ruzsa inequalities (e.g., \cite[Theorem 5.3]{GT7}), it follows that
$$\mathbb{E}_x \mathbb{E}_{h \in B(S, \rho)} 1_{H'}(S, \rho)(x + h) \ge (\delta/L)^{O(1)}.$$
By the pigeonhole principle, there exists a translation $x_0$ such that
$$\mathbb{E}_{h \in B(S, \rho)} 1_{H'}(x_0 + h) \ge (\delta/L)^{O(1)}.$$
Setting $A = \{h \in B(S, \rho), x_0 + h \in H'\}$, we find that $|A| \ge (\delta/L)^{O(1)}|B(S, \rho)|$ so direct insertion into the quadratic Fourier analytic arguments gives
$$\mathbb{E}_{h \in B(S, \rho)} |\mathbb{E}_{x \in G} f(x)\overline{f(x + h + x_0)} e(-(\xi(x_0) +2Mh)x)| \ge (\delta/L)^{O(1)}$$
where $\xi(x_0)$ is the phase factor in the exponent corresponding to the shift of $x_0$. 
\subsection{The Symmetry Argument}
The key is to showing that $Mx \cdot y$ and $My \cdot x$ are ``close together", so to speak, i.e., $e(Mx \cdot y - My \cdot x)$ has bounded approximate Fourier complexity. \\\\
We begin with the conclusion of the linearization step, opting to write it as
$$|\mathbb{E}_{h \in B(S_1, \rho_1), x \in B(S_2, \rho_2)} b_2(x + h)b_2(x)b_2(h) e(-2Mh \cdot x)| \ge (\delta/L)^{O(1)}$$
where $\rho_2$ is chosen to be much smaller than $\rho_1$. By the Cauchy-Schwarz inequality in the $h$ variable, we have
$$|\mathbb{E}_{h \in B(S_1, \rho_1), x, y \in B(S_2, \rho_2)} b_2(x + h)b_2(x)b_2(y + h)b_2(y)e(-Mh \cdot x)| \ge (\delta/L)^{O(1)}.$$
Making the change of variables $z = x + y + h$, we have
$$|\mathbb{E}_{x, y \in B(S_2, \rho_2), z \in x + y + B(S_1, \rho_1)} b_2(z - x) b_2(z - y) b_2(x)b_2(y) e(-2M(z - x - y) \cdot (y - x))| \ge (\delta/L)^{O(1)}.$$
Here, we observe that
$$M(z - x - y) \cdot (y - x) = Mx \cdot y - My \cdot x + f_1(x, z) + f_2(y, z)$$
where $f_1, f_2$ are functions. The point is that the phases $f_1$ and $f_2$ may be absorbed in $b_2$ and we are left with the phase $Mx \cdot y - My \cdot x$, which if studied measures how close $Mx \cdot y$ and $My \cdot x$ are, where ``closeness" means that they differ by an approximate degree one nilsequence. By the pigeonhole principle in $z$, we have (see discussion after \thref{linearization} for the definition of $b_4$ and more generally $b_j$)
$$|\mathbb{E}_{x, y \in B(S_2, \rho_2)} b_4(x)b_4(y) e(2(Mx \cdot y - My \cdot x))| \ge (\delta/L)^{O(1)}.$$
By Cauchy-Schwarz in $x$, and using the fact that $M$ is locally linear, it follows that
$$|\mathbb{E}_{y, y', x \in B(S_2, \rho_2)} b_8(y, y') e(2(Mx \cdot (y - y') - M(y - y') \cdot x))| \ge (\delta/L)^{O(1)}.$$
By the pigeonhole principle, we thus have
$$|\mathbb{E}_{x \in B(S_2, \rho_2)}  e(2(Mx \cdot (y - y') - M(y - y') \cdot x))| \ge (\delta/L)^{O(1)}$$
for all $y \in A$ with $|A| \ge (\delta/L)^{O(1)} |B(S_2, \rho_2)|$. Let $\{x, z\} = Mx \cdot z - Mz \cdot x$. Then the above inequality may be written as
$$|\mathbb{E}_{x \in B(S_2, \rho_2)}  e(2\{x, y - y'\})| \ge (\delta/L)^{O(1)}.$$
Applying \cite[lemma 8.4]{GT7}, and using the fact that $\{x, z\}$ is locally bilinear, it follows that
$$\|\{x, z\}\|_{\mathbb{R}/\mathbb{Z}} \le (\delta/L)^{-O(1)} \|x\|_{S_2}$$
for all $z \in 2A - 2A$ since $M$ is locally linear on $B(S, \rho)$. By \thref{localbogolyubov}, $2A - 2A$ contains a Bohr set $B(S_2 \cup S', \rho_2^{O(1)})$ (the set $B_3$ in our sketch above). Thus, it follows that for some Bohr set $B_3$, for $x, z$ in that Bohr set, we have
$$\|2(Mx \cdot z - Mz \cdot x)\|_{\mathbb{R}/\mathbb{Z}} \le (\delta/L)^{-O(1)} \|x\|_S.$$
To eliminate the factor of $2$, we simply multiply the phase set by $\frac{1}{2}$, and modify $B_3 = B(S_2 \cup S', \rho_2^{O(1)})$ via $B(\frac{1}{2}(S_2 \cup S'), \rho_2^{O(1)})$.
\subsection{Finishing the $U^3$ inverse theorem}
We can thus define $B(x, z) = Mx \cdot z + Mz \cdot x$. The claim is that $f$ correlates with the bilinear form $B$. By the conclusion of the linearization step, we have
$$|\mathbb{E}_{h \in B(S, \rho'), x \in G} b_2(h)b_2(x + h) \overline{f(x)}e(-Mh \cdot x)| \ge (\delta/L)^{O(1)}.$$
As $2Mh \cdot x = B(h, x) - \{x, h\}$ where $\{x, z\} = Mh \cdot x -Mx \cdot h$, it follows that if we localize $h$ to a small enough Bohr set, then $2Mh \cdot x \approx B(h, x)$. Thus, we use the pigeonhole principle and \thref{localization} to localize $h$ and $x$ to the Bohr set $B_4 = B(S_3, \rho_3^{O(1)})$, defined above:
$$\mathbb{E}_{y \in G} |\mathbb{E}_{h, x \in B_4} b_2(h)b_2(x + h)b_2(h, y) \overline{f(x + y)} e(-2Mh \cdot y)| \ge (\delta/L)^{O(1)}$$
and using the fact that $Mh \cdot x -Mx \cdot h \pmod{1}$ varies by at most  $(\delta/L)^{O(1)}$ on $B_4$, it follows that
$$\mathbb{E}_{y \in G} |\mathbb{E}_{h \in B_4, x \in G} b_4(h, y)b_2(x + h) \overline{f(x + y)} e(-B(h, x))| \ge (\delta/L)^{O(1)}.$$
Next, we use $B(h, x) = M(x + h) \cdot (x + h) - Mx \cdot x - Mh \cdot h$ and writing $\phi(x) = Mx \cdot x$ to obtain
$$\mathbb{E}_{y \in G} |\mathbb{E}_{h \in B_4, x \in G} b_4(h, y)b_2(x + h) \overline{f(x + y)} e(-\phi(x))| \ge (\delta/L)^{O(1)}.$$
Now localizing $x$ to $B_5 = B(S_4, \rho_4^{O(1)})$ once again so that $x + h$ lies in a Bohr set, we have
$$\mathbb{E}_{y \in G} |\mathbb{E}_{h \in B_4, x \in B_5} b_4(h, y) b_2(h + x) \overline{f(x + y)} e(-\phi(x))| \ge (\delta/L)^{O(1)}.$$
As the form $(h, x + h, x)$ is a complexity one form, we may use Cauchy-Schwarz and Plancheral to obtain
\begin{align*}
|\mathbb{E}_{h \in B_4, x \in B_5} b(h, y) b(h + x) \overline{f(x + y)} e(-\phi(x))| &\le \frac{1}{(\mathbb{E} 1_{B_4})(\mathbb{E} 1_{B_5})} |\mathbb{E}_{\xi} \hat{b}(-\xi, y) \hat{b}(\xi) \widehat{(f(\cdot + y) e(-\phi(\cdot)))}(\xi)| \\
&\le L^{O(1)}\frac{\mathbb{E} 1_{B_4 + B_5}}{(\mathbb{E} 1_{B_4})(\mathbb{E} 1_{B_5})} \sup_{\xi} |\mathbb{E}_{\xi \in G} f(x + y) e(-\phi(x) + \xi \cdot x)|.
\end{align*}
Hence,
$$\mathbb{E}_{y \in G} |\mathbb{E}_{x \in B_5} \overline{f(x + y)} e(-\phi(x) + \xi(y) \cdot x)| \ge (\delta/L)^{O(1)}.$$
This completes the proof of \thref{sandersinversetheorem}.
\section{Equidistribution of Multidimensional Polynomials}
Here, we shall record some results on the equidistribution theory of multidimensional polynomials, which is worked out in \cite{T1}.
\begin{proposition}\thlabel{multi}
Let $P(n_1, \dots, n_d) = \sum_{j \in [d]} \sum_{i_j \in [d_j]} a_{i_1, i_2, \dots, i_d} n_1^{i_1} n_2^{i_2} \cdots n_d^{i_d}$ be a multidimensional polynomial of dimension $d$. Let $I_i \subseteq [N_i]$ and suppose
$$\left|\sum_{n_i \in I_i \forall i} e(P(n))\right| \ge \delta N_1N_2 \cdots N_d.$$
Then either $N_j \ll_{d, d_1, \dots, d_d} \delta^{-O_{d, d_1, \dots, d_d}}$ for some $j$ or there exists $q \ll \delta^{-O_{d, d_1, \dots, d_d}}$ such that
$$\|q \alpha_{i_1, i_2, \dots, i_d}\|_{\mathbb{R}/\mathbb{Z}} \ll_{d, d_1, \dots, d_d} \delta^{-O_{d_1, \dots, d_d}}N_1^{-i_1}N_2^{-i_2} \cdots N_d^{-i_d}.$$
\end{proposition}

\begin{lemma}[Vinogradov's Lemma]\thlabel{Vinogradov}
Let $I \subseteq [N]$ be an interval and $P\colon \mathbb{Z} \to \mathbb{R}/\mathbb{Z}$ a polynomial of degree $d$ of the form $P(n) = \sum_{i = 0}^d \alpha_i n^i$. Suppose that $\|P(n)\| \le \epsilon$ for $\delta N$ many values of $n \in I$ with $0 < \delta, \epsilon < 1$. Then either
$$N \ll \delta^{-\exp(O(d)^{O(1)})}$$
or
$$\epsilon \ll O(\delta)^{\exp(O(d)^{O(1)})}$$
or there exists some $q \ll O(\delta)^{-\exp(O(d)^{O(1)})}$ such that
$$\|q\alpha_i \|_{\mathbb{R}/\mathbb{Z}} \ll \frac{\delta^{-O(1)}\epsilon}{N^i}.$$
\end{lemma}
The following is worked out in \cite[lemma A.12]{GT4}:
\begin{lemma}[Quadratic Vinogradov Lemma]\thlabel{QuadraticVinogradov}
Let $I \subseteq \mathbb{Z}$ be an interval and suppose
$$\{\ell \in I: \|\alpha \ell^2 + \beta \ell + \gamma\|_{\mathbb{R}/\mathbb{Z}} \le \epsilon\}$$
has size at least $\delta |I|$. Then either $\epsilon > \frac{1}{4} \delta$ or $|I| \ge 2^{58}\delta^{-12}$ or 
$$\|\alpha\|_{2^{43} \delta^{-9}, \mathbb{R}/\mathbb{Z}} \le 2^{141}\delta^{-28}|I|^{-2}.$$
\end{lemma}

\section{Properties of Gowers Uniformity Norms}
In this section, we will state without proof properties of Gowers uniformity norms. For a more comprehensive reference, see e.g., \cite{T2} and \cite{TV}. Let $G$ be a finite abelian group and $f\colon G \to \mathbb{C}$ a function. As defined in Section 2, the Gowers $U^{s}$ norm of $f$ is defined as
$$\|f\|_{U^s(G)}^{2^s} = \mathbb{E}_{n, h_1, h_2, \dots, h_s \in G} \prod_{\omega \in \{0, 1\}^s} C^{|\omega|}f(x + \omega \cdot h)$$
where $|\omega|$ denotes the number of ones in $\omega$ and $C$ denotes the conjugation operation $z \mapsto \bar{z}$. One can verify that the Gowers $U^s$ norms are seminorms and that for $s \ge 2$, the Gowers $U^s$ norms are indeed norms. The Gowers norms satisfy the below Cauchy-Schwarz-type inequality
\begin{lemma}[Cauchy-Schwarz-Gowers Inequality]\thlabel{CauchySchwarzGowers}
For $(f_\omega)_{\omega \in \{0, 1\}^s} \colon G \to \mathbb{C}^{\{0, 1\}^s}$, we define 
$$\langle (f_\omega)_{\omega \in \{0, 1\}^s} \rangle_{U^s(G)} := \mathbb{E}_{n, h_1, h_2, \dots, h_s \in G} \prod_{\omega \in \{0, 1\}^s} C^{|\omega|}f_\omega(x + \omega \cdot h).$$
Then
$$|\langle (f_\omega)_{\omega \in \{0, 1\}^s} \rangle_{U^s(G)} | \le \prod_{\omega \in \{0, 1\}^s} \|f_\omega\|_{U^s(G)}.$$
\end{lemma}
The Gowers norms also satisfy the following properties:
\begin{lemma}\thlabel{gowersproperties}
Let $f\colon G \to \mathbb{C}$. 
\begin{itemize}
    \item[1.] $\|f\|_{U^i(G)} \le \|f\|_{U^j(G)}$ whenever $i \le j$. 
    \item[2.] If $P \in \widehat{G}[x]$ has degree at most $s$, then $\|f e(P)\|_{U^s(G)} = \|f\|_{U^s(G)}$
    \item[3.] Then $\|f\|_{U^2(G)}^4 \le \|f\|_{L^2(G)}^2\|\hat{f}\|_{L^\infty(\hat{G})}^2$.
\end{itemize}
\end{lemma}
The last of these properties is known as the \emph{$U^2$ inverse theorem}. An analogous statement for the $U^3(\mathbb{Z}/N\mathbb{Z})$ norm is deduced in Appendix A.

\end{document}